\documentclass[11pt, reqno]{amsart}
\usepackage{amscd,amsmath,amssymb,amsthm,yfonts,mathrsfs,hhline,tikz}
\usetikzlibrary{arrows, decorations.markings, patterns}
\tikzset{>=latex}
\usepackage[matrix,arrow,curve]{xy}
\global\binoppenalty=10000

\newtheorem{theorem}{Theorem}[section]
\newtheorem{lemma}[theorem]{Lemma}
\newtheorem{prop}[theorem]{Proposition}
\newtheorem{cor}[theorem]{Corollary}
\newtheorem{conjecture}[theorem]{Conjecture}

\theoremstyle{definition}

\newtheorem{remark}[theorem]{Remark}
\newtheorem{example}[theorem]{Example}

\numberwithin{equation}{section}

\def\beq{\begin{equation}}
\def\eeq{\end{equation}}

\def\longra{\longrightarrow}

\newcommand{\hr}[1]{\left(#1\right)} 
\newcommand{\hm}[1]{\left|#1\right|} 
\newcommand{\ha}[1]{\left\langle#1\right\rangle} 
\newcommand{\hs}[1]{\left[#1\right]} 
\newcommand{\hc}[1]{\left\{#1\right\}} 

\def\le{\leqslant}
\def\ge{\geqslant}
\def\da{\downarrow}
\def\ua{\uparrow}

\def\Ac{\mathcal A}
\def\Ad{\operatorname{Ad}}

\def\bgt{\mathfrak b}

\def\C{\mathbb C}

\def\Dc{\mathcal D}
\def\Dgt{\mathfrak D}

\def\eps{\varepsilon}

\def\g{\mathfrak g}

\def\hgt{\mathfrak h}

\def\i{\mathbf i}

\def\Kc{\mathcal K}

\def\la{\lambda}
\def\La{\Lambda}
\def\Lb{\mathbf L}

\def\M{\mathrm M}

\def\Oc{\mathcal O}

\def\Rc{\mathcal R}

\def\sgn{\operatorname{sgn}}
\def\sl{\mathfrak{sl}}
\def\sla{\scalebox{0.7}{$\Lambda$}}

\def\V{\mathrm V}

\def\wh{\widehat}

\def\Wc{\mathcal W}

\def\Xc{\mathcal X}

\def\Z{\mathbb Z}
\def\Zc{\mathcal Z}

\begin{document}

\title[Cluster realization of $U_q(\mathfrak{sl_n})$]{Cluster realization of $U_q(\mathfrak{sl_n})$ from quantum character varieties}
\author{Gus Schrader}
\author{Alexander Shapiro}
\maketitle

\begin{abstract}
We construct an injective algebra homomorphism of the quantum group $U_q(\sl_{n+1})$ into a quantum cluster algebra $\mathbf{L}_n$ associated to the moduli space of framed $PGL_{n+1}$-local systems on a marked punctured disk. We obtain a description of the coproduct of $U_q(\sl_{n+1})$ in terms of the corresponding quantum cluster algebra associated to the marked twice punctured disk, and express the action of the $R$-matrix in terms of a mapping class group element corresponding to the half-Dehn twist rotating one puncture about the other. As a consequence, we realize the algebra automorphism of $U_q(\sl_{n+1})^{\otimes 2}$ given by conjugation by the $R$-matrix as an explicit sequence of cluster mutations, and derive a refined factorization of the $R$-matrix into quantum dilogarithms of cluster monomials. 
\end{abstract}

\section*{Introduction}

In~\cite{Fad99}, an intriguing realization of the quantum group $U_q(\sl_2)$ and the Drinfeld double of its Borel subalgebra was presented in terms of a quantum torus algebra $\Dgt$. Explicitly, the algebra $\Dgt$ has generators $\{w_1,w_2,w_3,w_4\}$, with the relations
\beq
\label{weyl-rels}
w_i w_{i+1}=q^{-2} w_{i+1} w_i \qquad\text{and}\qquad w_i w_{i+2} = w_{i+2} w_i
\eeq
where $i\in \mathbb{Z}/4\mathbb{Z}.$ In terms of the standard generators $ E, F, K,K'$ of the Drinfeld double (see Section 3 for the definitions), the embedding described in~\cite{Fad99} takes the form
\beq
\begin{aligned}
\label{faddeev-embed}
& E \mapsto \i (w_1+w_2), \qquad K \mapsto qw_2w_3, \\
& F \mapsto \i (w_3+w_4), \qquad K' \mapsto qw_4w_1,
\end{aligned}
\eeq
where $\i=\sqrt{-1}$. 

The embedding~\eqref{faddeev-embed} has some striking properties. First, as proposed in~\cite{Fad99}, one can use the Weyl-type relations~\eqref{weyl-rels} to define a modular double of $U_q(\sl_2)$ compatible with the regime $|q|=1$. Second, the image of the quasi $R$-matrix under this embedding admits a remarkable factorization into the product of four quantum dilogarithms:
\beq
\label{fad-R-fact}
\begin{aligned}
\bar\Rc = \Psi^q\hr{w_1 \otimes w_3} \Psi^q\hr{w_1 \otimes w_4} \Psi^q\hr{w_2 \otimes w_3} \Psi^q\hr{w_2 \otimes w_4}.
\end{aligned}
\eeq
These properties have been exploited in~\cite{PT99,PT01,BT03} to define and study a new category of $U_q(\sl_2)$-modules, the so-called \emph{positive representations}, which is closed under taking tensor products in the sense of a direct integral.

On the other hand, factorizations of the $U_q(\sl_2)$ quasi $R$-matrix of the form~\eqref{fad-R-fact} have also appeared in quantum Teichmuller theory. In~\cite{Kas01}, the action of the $R$-matrix is identified, up to a permutation, with an element of the mapping class group of the twice punctured disc. The mapping class group element in question corresponds to the half-Dehn twist rotating one puncture about the other. Having chosen certain triangulation of the twice punctured disc, this transformation can be decomposed into a sequence of four flips of the triangulation, as shown in Figure~\ref{fig-Dehn}. One is thus led to interpret each dilogarithm in the factorization~\eqref{fad-R-fact} as corresponding to a flip of a triangulation. In~\cite{HI14}, this observation was used to re-derive Kashaev's knot invariant.

In this paper, we explain how to generalize Faddeev's embedding~\eqref{faddeev-embed} to the case of the quantum group $U_q(\sl_{n+1})$ using the language of quantum cluster algebras. These quantum cluster algebras come in two types, often referred to as quantum $\Ac$- and $\Xc$-cluster algebras, whose precise definition we recall in Section~\ref{sect-cluster}. Roughly speaking, the cluster $\Ac$-coordinates are generalizations of the minimal collections of minors introduced by Fomin and Zelevinsky to test total positivity of an $n\times n$ matrix, while the cluster $\Xc$-coordinates are generalizations of the factorization parameters used to parameterize the locus of totally positive matrices. 

The study of the interplay between quantum groups and quantum cluster $\Ac$-algebras goes back to the original paper of Berenstein and Zelevinsky~\cite{BZ05} where the notion of a quantum cluster $\Ac$-algebra was defined.
In~\cite{BZ05}, it was shown that particular collections of generalized quantum minors constitute initial $\Ac$-clusters for the quantized algebras of functions $\Oc_q(G^{u,v})$ on double Bruhat cells in Kac-Moody groups. It was also conjectured that $\Oc_q(G^{u,v})$ is in fact isomorphic to the quantum upper cluster $\Ac$-algebra corresponding to these initial seeds. A version of this conjecture for the quantum unipotent coordinate ring $\Oc_q(N(w))$ associated to any Weyl group element $w$ was proved by Geiss, Leclerc and Schr\"oer in~\cite{GLS13} by studying certain subcategories of representations of the associated preprojective algebras. Moreover, it was shown in~\cite{GLS13} that cluster monomials from different clusters are linearly independent (unless equal) and belong to the Lusztig's semicanonical basis~\cite{Lus90}.

A new perspective on quantum cluster $\Ac$-algebras was suggested by Hernandez and Leclerc in~\cite{HL10,HL16}, where the notion of a monoidal categorification of a cluster algebra was introduced. The existence of a monoidal categorification guarantees that the cluster monomials are linearly independent and admit expansions in other clusters as Laurent polynomials with coefficients in $\Z_{\ge0}[q^{\pm\frac12}]$. It was shown in~\cite{HL10} that the finite-type cluster algebras for the~$A_{n}$ quivers are categorified by a certain subcategory of finite-dimensional modules of the quantum affine algebra $U_q(\wh\sl_{n+1})$, whose simple objects are the Kirillov-Reshetikhin modules. It was also conjectured that cluster algebras for arbitrary Dynkin quivers should admit monoidal categorifications. This conjecture was proven in the simply-laced case by Nakajima~\cite{Nak11} using the category of perverse sheaves on graded quiver varieties. The result of~\cite{Nak11} was further generalized by Kimura and Qin in~\cite{KQ14} where they showed by similar techniques that any cluster algebra with an acyclic seed admits a monoidal categorification. Finally, in the recent work~\cite{KKKO18} of Kang, Kashiwara, Kim and Oh, a monoidal categorification of the cluster structure on $\Oc_q(N(w))$ was constructed via categories of modules over symmetric KLR algebras. As a consequence, they obtained a proof of Kimura's conjecture~\cite{Kim12} that every cluster monomial belongs to the dual canonical basis of $\Oc_q(N(w))$. The latter conjecture is a refined version of (part of) a conjecture made by Berenstein and Zelevinsky~\cite{BZ93} and corrected later by Fomin and Zelevinsky in~\cite{FZ02}, which relates cluster structure and the dual canonical basis of $\Oc_q(N)$. We refer the reader interested in these developments to the recent survey by Kashiwara~\cite{Kas18}.

Relations between quantum groups and the quantum cluster $\Xc$-algebras has been studied somewhat less extensively. It is known that quantum cluster $\Xc$-algebra structures on quantum nilpotent algebras $U_q(\mathfrak{n}(w))$ are provided by the Feigin homomorphisms~\cite{Ber96,Rup15}, in which the quantum cluster $\Xc$-coordinates can be regarded as quantum analogs of factorization parameters for unipotent cells, see~\cite{BFZ96} and~\cite{FG06b}

It has remained unclear, however, (in either $\Ac$- or $\Xc$-algebra setup) whether there exists a quantum cluster algebra realization of the whole quantized enveloping algebra $U_q(\g)$ associated to a semisimple Lie algebra $\g$. The existence of such a cluster $\Ac$-realization is suggested by the work of Gekhtman, Shapiro, and Vainshtein \cite{GSV18} in which the algebra of functions on the Poisson-Lie dual group $GL_n^*$ ({\it cf} Remark~\ref{normalization-remark}) was shown to carry the structure of generalized cluster $\Ac$-algebra, as well as by Qin's construction \cite{Qin16} of $U_q(\g)$ for simply-laced~$\g$ as the quotient of a Grothendieck ring arising from certain cyclic quiver varieties.

In the present work, we show that a cluster $\Xc$-realization of $U_q(\g)$ indeed exists when $\g=\sl_{n+1}$. Our approach to constructing this realization is based on the quantum cluster structure associated to moduli spaces of $PGL_{n+1}$-local systems on marked surfaceы, see~\cite{FG06a,FG09}. Cluster charts on these varieties are obtained from an ideal triangulation of the surface by ``gluing'' certain simpler cluster charts associated to each triangle. In the case of moduli spaces of $PGL_{n+1}$-local systems, a flip of a triangulation can be realized as sequences of $\binom{n+2}{3}$ cluster mutations.

Taking a particular cluster chart on the moduli space associated to the regular triangulation of the punctured disk (see Section 2 for the definition), we obtain by this gluing procedure a quiver and a corresponding quantum cluster $\Xc$-algebra~$\Lb_n$, the quantized algebra of regular functions on the corresponding cluster $\Xc$-variety. Our first main result, Theorem~\ref{thm-embed}, provides an explicit embedding of $U_q(\sl_{n+1})$ into $\Lb_n$. This embedding has an interesting property that each Chevalley generator of $U_q(\sl_{n+1})$ is a cluster monomial in some quantum cluster chart. As explained by Goncharov and Shen~\cite{GS16}, there exists an action of the Weyl group $S_{n+1}$ on $\Lb_n$ by cluster transformations, and we conjecture that the image of $U_q(\sl_{n+1})$ under our embedding coincides precisely with the subalgebra of invariants for this action. In the simplest case, $n=1$, our result reproduces Faddeev's realization~\eqref{faddeev-embed} of $U_q(\sl_2)$ in terms of the quantum torus $\Dc_1$ associated to the cyclic quiver with four nodes (see Figure~\ref{fig-A1}). Moreover, since the quantum group is contained in the quantum cluster $\Xc$-algebra $\Lb_n$, one can apply certain sequences of cluster mutations to obtain embeddings of $U_q(\sl_{n+1})$ into quantum tori corresponding to the two self-folded triangulations of the punctured disk shown in Figure~\ref{fig-triangulations}. Each of these triangulations turns out to have the property that the restriction of the corresponding embedding to one of the quantum Borel subalgebras $U_q(\bgt_\pm)$ coincides with the Feigin homomorphism. 

We also solve the problem of describing the coproduct and the universal $R$-matrix of $U_q(\sl_{n+1})$ in cluster-algebraic terms. We formulate this description in terms of another quantum cluster algebra, this time corresponding to a quiver built from a triangulation of the twice punctured disk. As we explain in Remark~\ref{rem-comult}, the coproduct admits a simple description in terms of cluster variables.

We finish this paper by proving in Theorem~\ref{thm-main} that the automorphism $P\circ\Ad_{\Rc}$ of~$U_q(\sl_{n+1})^{\otimes 2}$ given by conjugation by the $R$-matrix followed by the flip of tensor factors can be identified with a cluster transformation given as the composite of the half-Dehn twist and a certain permutation. In the course of the proof, we obtain in Theorem~\ref{thm-factor} a refined factorization of $\Rc$ into a product of $4\binom{n+2}{3}$ quantum dilogarithms, one for each mutation required to achieve the half-Dehn twist realized as a sequence of four flips. In the case of $U_q(\sl_2)$, when each flip can be achieved by a single cluster mutation, we recover Faddeev's factorization~\eqref{fad-R-fact}.

Let us conclude the introduction by describing some applications and future directions suggested by our results. First, as explained in~\cite{FG09}, the quantized algebra of functions on a cluster $\Xc$-variety comes equipped with a natural family of Hilbert space representations on which the groupoid of cluster transformations acts by unitary operators. In particular, one can restrict the Hilbert space representations of the quantum cluster algebra $\Lb_n$ to the subalgebra $U_q(\sl_{n+1})\subset \Lb_n$. The quantum group representations obtained in this fashion turn out to be isomorphic to the \emph{positive representations} defined in~\cite{FI14}, which are higher rank generalizations of the representations of $U_q(\sl_2)$ introduced by Ponsot and Teschner in~\cite{PT99,PT01}. Thus, the cluster transformations for $\Lb_n$ (for example, those corresponding to flips of diagonals in an ideal triangulation) provide a useful topological lens through which to investigate the rather intricate algebraic structure of these representations. More specifically, in~\cite{SS17b,SS18} we use this perspective to prove the conjecture of Frenkel and Ip that the positive representations of $U_q(\sl_{n+1})$ are closed under tensor product. Moreover, in a work in preparation with Ivan Ip we show that the cluster-algebraic approach can be used to generalize the results of~\cite{Ip13} and construct a $C^*$-Hopf algebra $L^2\hr{SL_q^+(n+1,\mathbb{R})}$ satisfying a positive analog of the Peter-Weyl theorem for the modular double of $U_q(\sl_{n+1},\mathbb{R})$. We would also like to note that since the first version of this paper appeared, its results were generalized in~\cite{Ip18} to other Dynkin types, which allows to use techniques devised by us in~\cite{SS17b} to study positive representations of quantum groups beyond Dynkin type $A$.

Second, our proof of the injectivity of the homomorphism $U_q(\sl_{n+1}) \to \Lb_n$ in Proposition~\ref{prop-inj} is based on ``tropicalizing'' the images of PBW basis elements of $U_q(\sl_{n+1})$. As we observe in Remark~\ref{new-hive-rmk}, the condition that a cluster monomial $\vec{X}$ appears as the leading term of a PBW basis element is equivalent to the tropicalized Goncharov-Shen potential function $\chi_0^t$ having non-negative value at $\vec{X}$. The non-negativity of this tropical potential is closely related to the \emph{hive condition}, introduced by Knutson and Tao in~\cite{KT99} in their combinatorial reformulation of the Littlewood-Richardson rule and proof of the saturation conjecture for $GL_n$. We find it an intriguing problem to further investigate this connection to the Littlewood-Richardson rule and the decomposition of tensor products of \emph{finite dimensional} quantum group representations, especially in light of the application of the cluster realization of $U_q(\sl_{n+1})$ to the problem of decomposing tensor products of its \emph{positive} representations.

Finally, the fact that in an appropriate cluster each Chevalley generator of $U_q(\sl_{n+1})$ becomes a {\em cluster monomial} under our embedding appears interesting from the point of view of categorification of quantum groups. Indeed, in the setup of monoidal categorification, each cluster monomial corresponds to a certain simple object in the corresponding category. In the recent paper~\cite{FT17} it has been shown that the quantum group $U_q(\sl_{n+1})$ can be realized as a quantum $K$-theoretic Coulomb branch~\cite{BFN16} for a certain $A_n$ quiver gauge theory. This suggests the possibility of categorifying $U_q(\sl_{n+1})$ via an appropriate category of equivariant coherent sheaves on the variety of triples corresponding to this quiver gauge theory. Along similar lines, Cautis and Williams~\cite{CW18} have recently proposed a monoidal categorification of the quantum cluster algebra associated to the type $A_n$ quantum Coxeter-Toda chain using the category of equivariant perverse coherent sheaves on another Coulomb branch, the affine Grassmannian $\mathrm{Gr}_{GL_{n+1}}$. In the future, we hope to return to the problem of constructing a categorification of $U_q(\sl_{n+1})$ based on a monoidal categorification of the corresponding cluster algebra.

The article is organized as follows. In Section~\ref{sect-cluster}, we recall some basic facts about quantum cluster algebras and the quantum dilogarithm function. Section~\ref{sect-conf} reviews the quantum character varieties as defined in~\cite{FG06a}. We also recall the procedure of quiver amalgamation and explain how a flip of the triangulation can be realized as a sequence of cluster transformations. In Section~\ref{sect-qgroups}, we fix our notations and conventions regarding the quantum group $U_q(\sl_{n+1})$. We state our first main result in Section~\ref{sect-embed}: an explicit embedding of $U_q(\sl_{n+1})$ into a quantum cluster algebra $\Lb_n$ built from the punctured disk. Section~\ref{sect-twist} recalls the combinatorial description of the half-Dehn twist on a twice punctured disk. In Section~\ref{sect-R}, we prove our next main result on the cluster nature of the $R$-matrix, while its refined factorization appears in Section~\ref{sect-factor}. We conclude the paper by comparing our results to those of Faddeev in Section~\ref{sect-example}.

\section*{Acknowledgements}
We would like to thank Vladimir Fock for sharing his insights and expertise, and for his suggestion to study quantum groups via quantization of the moduli spaces of framed local systems considered in his work with Alexander Goncharov. We are grateful to Arkady Berenstein for many inspiring discussions. Our gratitude also goes to Nicolai Reshetikhin for his support and helpful suggestions throughout the course of this work. We thank Linhui Shen for careful reading of the first version of this manuscript, for teaching us a trick used in the proof of Proposition~\ref{prop-upper}, and for providing many excellent remarks, which led in particular to a significant simplification in the exposition of Lemma 7.6. We are grateful to David Hernandez for his hospitality in Paris, where the first version of this paper was completed. Finally, we would like to thank the anonymous referee for valuable comments which helped us improve exposition of the paper.

\section{Quantum cluster algebras}
\label{sect-cluster}

In this section we recall a few basic facts about cluster algebras and their quantization following~\cite{FG06a}. We shall need only skew-symmetric exchange matrices, and we incorporate this in the definition of a cluster seed.

A \emph{seed} $\Sigma$ is triple $(I,I_0,\eps)$ where $I$ is a finite set, $I_0\subset I$ is a subset, and $\eps = (\eps_{ij})_{i,j \in I}$ is a skew-symmetric\footnote{in general, the matrix $\eps$ is allowed to be skew-symmetrizable.} $\frac{1}{2}\Z$-valued matrix, such that $\eps_{ij} \in \Z$ unless $i,j \in I_0$. To a seed $\Sigma$ we associate a pair of algebraic tori $\Ac_\Sigma = (\C^\times)^{|I|}$ and $\Xc_\Sigma = (\C^\times)^{|I|}$, equipped respectively with coordinates $\{A_1,\ldots,A_{|I|}\}$ and $\{X_1,\ldots,X_{|I|}\}$. We refer to the torus $\Ac_\Sigma$ as the \emph{cluster $\Ac$-torus}, the torus $\Xc_\Sigma$ as the \emph{cluster $\Xc$-torus}, and the matrix $\eps$ as the \emph{exchange matrix} associated to the seed $\Sigma$. The coordinates $A_i$ and $X_i$ are called \emph{cluster $\Ac$-variables} and \emph{cluster $\Xc$-variables} respectively, and they are said to be \emph{frozen} if $i \in I_0$. Let $M$ be the $I \times I$ matrix with $M_{ij} = 0$ unless both $i$ and $j$ are frozen, and such that $\widetilde{\eps} = \eps+M$ is an integer matrix. Then there is a regular map $p^M_\Sigma \colon \Ac_\Sigma\to\Xc_\Sigma$ called the \emph{cluster ensemble map} corresponding to the matrix $M$, defined by the formula
$$
\hr{p_\Sigma^M}^*X_k = \prod_{i\in I}A_i^{\widetilde{\eps}_{ki}}. 
$$

Given a pair of seeds $\Sigma = (I, I_0, \eps)$, $\Sigma' = (I', I'_0, \eps')$, and an element $k \in I \setminus I_0$ we say that an isomorphism $\mu_k \colon I \to I'$ is a \emph{seed mutation in direction $k$} if $\mu_k(I_0) = I'_0$ and
\beq
\label{mut-mat}
\eps'_{\mu_k(i), \mu_k(j)} =
\begin{cases}
-\eps_{ij} &\text{if} \; i=k \;\text{or}\; j=k, \\
\eps_{ij} &\text{if} \; \eps_{ik}\eps_{kj} \le 0, \\
\eps_{ij} + |\eps_{ik}|\eps_{kj} &\text{if} \; \eps_{ik}\eps_{kj}>0.
\end{cases}
\eeq

The combinatorial data of a cluster seed can be conveniently encoded by a quiver with vertices $\{v_i\}$ labelled by elements of the set $I$ and with adjacency matrix $\eps$. The arrows $v_i \to v_j$ between a pair of frozen variables are considered to be weighted by $\eps_{ij}$. Then the mutation $\mu_k$ of the corresponding quiver can be performed in three steps:
\begin{enumerate}
\item[1)] reverse all the arrows incident to the vertex $k$;
\item[2)] for each pair of arrows $k \to i$ and $j \to k$ draw an arrow $i \to j$;
\item[3)] delete pairs of arrows $i \to j$ and $j \to i$ going in the opposite directions.
\end{enumerate}

To a seed mutation $\mu_k$ is associated a pair of birational isomorphisms of cluster tori $\mu_k^{\Ac} \colon \Ac_\Sigma \to \Ac_{\Sigma'}$ and $\mu_k^{\Xc} \colon \Xc_\Sigma \to \Xc_{\Sigma'}$ defined as follows:

\begin{align}
\label{A-mut-cl}
&\hr{\mu_k^\Ac}^\ast A_{\mu_k(i)} = 
\begin{cases}
A_k^{-1}\hr{\prod_{i|\eps_{ki}>0}A_i^{\eps_{ki}} + \prod_{i|\eps_{ki}<0}A_i^{-\eps_{ki}}} &\text{if} \; i = k, \\
A_i &\text{if} \; i\ne k,
\end{cases}
\\
\label{X-mut-cl}
&\hr{\mu_k^\Xc}^\ast X_{\mu_k(i)} = 
\begin{cases}
X_k^{-1} &\text{if} \; i=k, \\
X_i\hr{1+X_k^{-\sgn(\eps_{ki})}}^{-\eps_{ki}} &\text{if} \; i \ne k.
\end{cases}
\end{align}
The isomorphisms~\eqref{A-mut-cl} and~\eqref{X-mut-cl} are referred to as \emph{cluster $\Ac$-} and \emph{cluster $\Xc$-mutations} respectively. As observed in~\cite[Proposition 4.7]{Wil13}, the two kinds of mutations are intertwined by the cluster ensemble morphisms $p^M_{\i}$: one has 
\beq
\label{ensemble-map}
\mu^{\Xc}_k \circ p^M_\Sigma = p^M_{\mu_k(\Sigma)}\circ\mu^{\Ac}_k.
\eeq

The \emph{cluster algebra} associated to the seed $\Sigma$ is defined to be the subring of the fraction field of $\Oc(\Ac_\Sigma)$ generated by cluster variables from all seeds mutation equivalent to $\Sigma$. The \emph{cluster $\Ac$-variety} $\Ac$ corresponding to a seed $\Sigma$ is defined to be the affine scheme whose ring of regular functions $\Oc(\Ac)$ consists of all \emph{universally Laurent} elements, i.e. the elements $f\in \Oc(\Ac_\Sigma)$ that remain Laurent polynomials under all finite sequences of cluster $\Ac$-mutations. The ring of regular functions $\Oc(\Ac_\Sigma)$ is sometimes referred to as the {\em upper cluster algebra}. By the Laurent phenomenon property~\cite{FZ02}, the upper cluster algebra $\Oc(\Ac)$ always contains the cluster algebra. 

The \emph{cluster $\Xc$-variety} $\Xc$ is defined similarly, using the algebra $\Oc(\Xc_\Sigma)$ and the cluster $\Xc$-mutations. In view of equality~\eqref{ensemble-map}, for each choice of the matrix $M$ there is a well-defined morphism $p^M \colon \Ac \to \Xc$ which is an isomorphism when $\det(\widetilde{\eps})=\pm 1$. 

The algebra of regular functions on each cluster $\Xc$-torus comes equipped with a Poisson structure defined by
$$
\hc{X_i, X_j} = 2\eps_{ij}X_iX_j, \qquad i,j \in I.
$$
These Poisson structures have the property that the cluster $\Xc$-mutations are Poisson maps, so that the cluster $\Xc$-variety is equipped with a canonically defined Poisson structure. 

The algebra of functions $\Oc(\Xc_\Sigma)$ admits a quantization $\Xc_\Sigma^q = \Oc_q(\Xc_\Sigma)$ called the \emph{quantum torus algebra} associated to the seed $\Sigma$. It is the algebra over $\Z[q^{\pm\frac{1}{2}}]$ defined by generators $X_i^{\pm 1}$, $i \in I$, subject to relations
$$
X_i X_j = q^{2\eps_{ji}}X_jX_i.
$$

For each cluster mutation $\mu_k$ there is a homomorphism $\mu_k^q \colon \Xc_\Sigma^q \to \Xc_{\Sigma'}^q$ called the \emph{quantum cluster mutation}, defined by
\beq
\label{X-mut}
\mu_k^q(X_i) = 
\begin{cases}
X_k^{-1}, & \text{if} \; i = k, \\[8pt]
X_i \prod\limits_{r=1}^{\eps_{ki}}\hr{1 + q^{2r-1}X_k^{-1}}^{-1}, & \text{if} \; i \ne k \; \text{and} \; \eps_{ki} \ge 0, \\
X_i \prod\limits_{r=1}^{-\eps_{ki}}\hr{1 + q^{2r-1}X_k}, & \text{if} \; i \ne k \; \text{and} \; \eps_{ki} \le 0.
\end{cases}
\eeq
Here the generators $X_i$ on the right hand side belong to the torus $\Xc_{\Sigma'}^q$. The \emph{quantum cluster $\Xc$-algebra} $\Lb_\Xc$ associated to a seed $\Sigma$ is defined to be the subalgebra of $\Oc_q(\Xc_\Sigma)$ consisting of universally Laurent elements. This notion of quantum cluster $\Xc$-algebra was introduced by Fock and Goncharov in~\cite{FG09}. When $\det(\widetilde{\eps})=\pm 1$, it is related to the notion of quantum cluster algebras introduced by Berenstein and Zelevinsky~\cite{BZ05} as follows, {\it cf.}~\cite[Section 2.7]{FG09}. Denoting by $\lambda$ the matrix inverse to $\widetilde{\eps}$, to each seed $\Sigma$ one associates a quantum torus algebra $\Oc_q(\Ac_\Sigma)$ generated by $A_i, i\in I$, with relations $A_iA_j = q^{2\lambda_{ji}}A_jA_i$. Mutations between these tori can then be defined by means of the rule given by formula~(4.23) in~\cite{BZ05}. Then the algebra $\Lb_{\Ac}$ of elements in $\Oc_q(\Ac_\Sigma)$ that are universally Laurent with respect to this mutation rule contains as a subalgebra the quantum cluster algebra of Berenstein and Zelevinsky. Moreover, in each seed $\Sigma$ one has an isomorphism 
$$
p_\Sigma^* \colon \Oc_q(\Xc_\Sigma) \to \Oc_q(\Ac_\Sigma), \qquad p_{\Sigma}^*X_i = q^{\sum_{k\leq l}\la_{kl}\widetilde{\eps}_{ik}\widetilde{\eps}_{il}}A_k^{\widetilde{\eps}_{ik}}.
$$

The quantum cluster $\Xc$-mutation $\mu_k^q$ defined in~\eqref{X-mut} can be written as a composition of two homomorphisms, namely
$$
\mu_k^q = \mu_k^\sharp \circ \mu_k'
$$
where $\mu_k' \colon \Xc_\Sigma^q \to \Xc_{\Sigma'}^q$ is a monomial transformation defined by
$$
X_i \longmapsto
\begin{cases}
X_k^{-1}, & \text{if} \; i = k, \\
q^{\eps_{ik}\eps_{ki}}X_iX_k^{\eps_{ki}}, & \text{if} \; i \ne k \; \text{and} \; \eps_{ki} \ge 0, \\
X_i, & \text{if} \; i \ne k \; \text{and} \; \eps_{ki} \le 0,
\end{cases}
$$
and 
$$
\mu_k^\sharp = \Ad_{\Psi^q\hr{X_k}}
$$
is a conjugation by the quantum dilogarithm function
$$
\Psi^q(x) = \frac1{(1+qx)(1+q^3x)(1+q^5x)\dots}.
$$
Mutation of the exchange matrix is incorporated into the monomial transformation $\mu'_k$. The following lemma will prove very useful in what follows.

\begin{lemma}\cite[Proposition 2.4]{GS16}
\label{mut-decomp}
A sequence of mutations $\mu_{i_k}^q \dots \mu_{i_1}^q$ can be written as follows
$$
\mu_{i_k}^q \dots \mu_{i_1}^q = \Phi_k \circ \M_k
$$
where
$$
\Phi_k = \Ad_{\Psi^q\hr{X_{i_1}}} \Ad_{\Psi^q\hr{\mu'_{i_1}\hr{X_{i_2}}}} \dots \Ad_{\Psi^q\hr{\mu'_{i_{k-1}} \dots \mu'_{i_1}\hr{X_{i_k}}}}
$$
and
$$
\M_k = \mu'_{i_k} \dots \mu'_{i_2} \mu'_{i_1}.
$$
\end{lemma}

\begin{proof}
We shall prove the lemma by induction. Assume the statement holds for some $k=r-1$. Then
$$
\mu_{i_r}^q \dots \mu_{i_1}^q = \Ad_{\Psi^q\hr{\Phi_{r-1}\hr{\M_{r-1}\hr{X_{i_r}}}}} \mu'_{i_r} \Phi_{r-1} \M_{r-1}.
$$
Now the proof follows from the fact that the homomorphisms $\mu'_{i_r}$ and $\Phi_{r-1}$ commute, and the following relation:
$$
\Ad_{\Psi^q\hr{\Phi_{r-1}\hr{\M_{r-1}\hr{X_{i_r}}}}} \Phi_{r-1} = \Phi_{r-1} \Ad_{\Psi^q\hr{\M_{r-1}\hr{X_{i_r}}}} = \Phi_r.
$$
\end{proof}

We conclude this section with the two properties of the quantum dilogarithm which we will use liberally throughout the paper. For any $u$ and~$v$ such that $uv = q^{-2}vu$ we have
\begin{align}
\label{exp-prod}
&\Psi^q(u) \Psi^q(v) = \Psi^q(u+v) \\
\label{pentagon}
&\Psi^q(v) \Psi^q(u) = \Psi^q(u)\Psi^q(qvu)\Psi^q(v)
\end{align}
The first equality is nothing but a $q$-analogue of the addition law for exponentials, while the second one is known as the \emph{pentagon identity}.

\section{Quantum character varieties}
\label{sect-conf}

Let $\wh S$ be a decorated surface --- that is, a topological surface $S$ with non-empty boundary $\partial S$, equipped with a finite collection of marked points $x_1,\ldots, x_r\in \partial S$ and punctures $p_1,\ldots, p_s$. In~\cite{FG06a}, the moduli space $\Xc_{\wh S,PGL_m}$ of $PGL_m$-local systems on $S$ with reductions to Borel subgroups at each marked point $x_i$ and each puncture $p_i$, was defined and shown to admit the structure of a cluster $\Xc$-variety. In particular, suppose that~$T$ is an \emph{ideal triangulation} of~$\wh S$, that is all vertices of~$T$ are at either marked points or punctures; {\it e.g.} all ideal triangulations of a punctured disk with a pair of marked points on the boundary are drawn in Figure~\ref{fig-triangulations}. It was shown in~\cite{FG06a} that for each ideal triangulation, there exists a cluster $\Xc$-chart on $\Xc_{\wh S,PGL_m}$. We shall now recall the construction of such a chart together with its canonical quantization.

The first step is to describe the quantum cluster $\Xc$-chart associated to a single triangle. To do this, consider a triangle $ABC$ given by the equation $x+y+z=m$ in the octant $x,y,z\ge0$, and intersect it with planes $x=p$, $y=p$, and $z=p$ for all $0 < p < m$, $p \in \Z$. The resulting picture is called the \emph{$m$-triangulation} of the triangle~$ABC$. Let us now color the triangles of the $m$-triangulation in black and white as in Figure~\ref{fig-triang}, so that triangles adjacent to vertices $A$, $B$, or $C$ are black, and two triangles sharing an edge are of different color. Now, we orient the edges of white triangles counterclockwise. Finally, we connect the vertices of the $m$-triangulation lying on the same side of the triangle $ABC$ by dashed arrows in the clockwise direction. The resulting graph is shown in Figure~\ref{fig-triang}. Note that the vertices on the boundary of $ABC$ are depicted by squares. Throughout the text we will use square vertices for frozen variables. All dashed arrows will be of weight $\frac12$, that is a dashed arrow $v_i \to v_j$ denotes the commutation relation $X_iX_j = q^{-1}X_jX_i$.

\begin{figure}[h]
\centering
\begin{tikzpicture}[thick, y=0.866cm, x=0.5cm]

\node at (-0.5,-0.1) {A};
\node at (4,4.4) {B};
\node at (8.5,-0.1) {C};

\draw (0,0) to (8,0);
\draw (1,1) to (7,1);
\draw (2,2) to (6,2);
\draw (3,3) to (5,3);

\draw (0,0) to (4,4);
\draw (2,0) to (5,3);
\draw (4,0) to (6,2);
\draw (6,0) to (7,1);

\draw (4,4) to (8,0);
\draw (3,3) to (6,0);
\draw (2,2) to (4,0);
\draw (1,1) to (2,0);

\fill[pattern=north east lines] (0,0) to (1,1) to (2,0);
\fill[pattern=north east lines] (2,0) to (3,1) to (4,0);
\fill[pattern=north east lines] (4,0) to (5,1) to (6,0);
\fill[pattern=north east lines] (6,0) to (7,1) to (8,0);

\fill[pattern=north east lines] (1,1) to (2,2) to (3,1);
\fill[pattern=north east lines] (3,1) to (4,2) to (5,1);
\fill[pattern=north east lines] (5,1) to (6,2) to (7,1);

\fill[pattern=north east lines] (2,2) to (3,3) to (4,2);
\fill[pattern=north east lines] (4,2) to (5,3) to (6,2);

\fill[pattern=north east lines] (3,3) to (4,4) to (5,3);

\end{tikzpicture}
\qquad\qquad
\begin{tikzpicture}[every node/.style={inner sep=0, minimum size=0.35cm, thick}, thick, y=0.866cm, x=0.5cm]

\node at (2,-0.2) {};

\node (1) at (2,0) [draw] {\tiny{2}};
\node (2) at (4,0) [draw] {\tiny{5}};
\node (3) at (6,0) [draw] {\tiny{9}};
\node (4) at (1,1) [draw] {\tiny{1}};
\node (5) at (3,1) [circle, draw] {\tiny{4}};
\node (6) at (5,1) [circle, draw] {\tiny{8}};
\node (7) at (7,1) [draw] {\tiny{12}};
\node (8) at (2,2) [draw] {\tiny{3}};
\node (9) at (4,2) [circle, draw] {\tiny{7}};
\node (10) at (6,2) [draw] {\tiny{11}};
\node (11) at (3,3) [draw] {\tiny{6}};
\node (12) at (5,3) [draw] {\tiny{10}};

\draw[->, dashed] (2) to (1);
\draw[->, dashed] (3) to (2);
\draw[->] (3) to (7);
\draw[->, dashed] (10) to (7);
\draw[->, dashed] (12) to (10);
\draw[->] (12) to (11);
\draw[->, dashed] (8) to (11);
\draw[->, dashed] (4) to (8);
\draw[->] (4) to (1);

\draw[->] (9) to (12);
\draw[->] (5) to (9);
\draw[->] (1) to (5);
\draw[->] (6) to (10);
\draw[->] (2) to (6);

\draw[->] (6) to (3);
\draw[->] (9) to (6);
\draw[->] (11) to (9);
\draw[->] (5) to (2);
\draw[->] (8) to (5);

\draw[->] (7) to (6);
\draw[->] (6) to (5);
\draw[->] (5) to (4);
\draw[->] (10) to (9);
\draw[->] (9) to (8);

\end{tikzpicture}
\caption{Cluster $\Xc$-coordinates on the configuration space of 3 flags and 3 lines.}
\label{fig-triang}
\end{figure}

Now, let us recall the procedure of \emph{amalgamation} of two quivers by a subset of frozen variables, following~\cite{FG06b}. In simple words, amalgamation is nothing but the gluing of two quivers by a number of frozen vertices. More formally, let $Q_1$, $Q_2$ be a pair of quivers, and $I_1$, $I_2$ be certain subsets of frozen variables in $Q_1$, $Q_2$ respectively. Assuming there exists a bijection $\phi \colon I_1 \to I_2$ we can amalgamate quivers $Q_1$ and $Q_2$ by the subsets $I_1$, $I_2$ along $\phi$. The result is a new quiver $Q$ constructed in the following two steps:
\begin{enumerate}
\item[1)] for any $i \in I_1$ identify vertices $v_i \in Q_1$ and $v_{\phi(i)} \in Q_2$ in the union $Q_1 \sqcup Q_2$;
\item[2)] for any pair $i,j \in I_1$ with an arrow $v_i \to v_j$ in $Q_1$ labelled by $\eps_{ij}$ and an arrow $v_{\phi(i)} \to v_{\phi(j)}$ in $Q_2$ labelled by $\eps_{\phi(i),\phi(j)}$, label the arrow between corresponding vertices in $Q$ by $\eps_{ij} + \eps_{\phi(i),\phi(j)}$
\end{enumerate}
Amalgamation of a pair of quivers $Q_1$, $Q_2$ into a quiver $Q$ induces an embedding $\Xc^q_\Sigma \to \Xc^q_{\Sigma_1} \otimes \Xc^q_{\Sigma_2}$ of the corresponding cluster $\Xc$-tori:
$$
X_i \mapsto
\begin{cases}
X_i \otimes 1, &\text{if} \; i \in Q_1 \setminus I_1, \\
1 \otimes X_i, &\text{if} \; i \in Q_2 \setminus I_2, \\
X_i \otimes X_{\phi(i)}, &\text{otherwise.}
\end{cases}
$$

An example of amalgamation is shown in Figure~\ref{fig-flip}. There, the left quiver is obtained by amalgamating a triangle $ABC$ from Figure~\ref{fig-triang} with a similar triangle along the side $BC$ (or more precisely, along frozen vertices 10, 11, and 12 on the edge $BC$). Another example is shown in Figure~\ref{fig-A3} where a triangle $ABC$ is now amalgamated by 2 sides. Finally, the process of amalgamation is best visible in Figure~\ref{fig-amalg}.

As explained in~\cite{FG06a}, in order to construct the cluster $\Xc$-coordinate chart on $\Xc_{\wh S,PGL_m}$
corresponding to an ideal triangulation $T$ of $\wh S$, one implements the following procedure:
\begin{enumerate}
\item[1)] $m$-triangulate each of the ideal triangles in $T$;
\item[2)] for any pair of ideal triangles in $T$ sharing an edge, amalgamate the corresponding pair of quivers by this edge.
\end{enumerate}
In general, different ideal triangulations of $\wh S$ result in different quivers, and hence different cluster $\Xc$-tori. However, any triangulation can be transformed into any other by a sequence of \emph{flips} that replace one diagonal in an ideal 4-gon with the other one. Each flip corresponds to the following sequence of cluster mutations that we shall recall on the example shown in Figure~\ref{fig-flip}. There, a flip is obtained in three steps. First, mutate at vertices 10, 11, 12, second, mutate at vertices 7, 8, 14, 15, and third, mutate at vertices 4, 11, 18. Note, that the order of mutations within one step does not matter. In general, a flip in an $m$-triangulated 4-gon consists of $m-1$ steps. On the $i$-th step, one should do the following. First, inscribe an $i$-by-$(m-i)$ rectangle in the 4-gon, such that vertices of the rectangle coincide with boundary vertices of the $m$-triangulation and the side of the rectangle of length $m-i$ goes along the diagonal of a 4-gon. Second, divide the rectangle into $i (m-i)$ squares and mutate at the center of each square. As in the example, the order of mutations within a single step does not matter.

\begin{figure}[h]
\centering
\begin{tikzpicture}[every node/.style={inner sep=0, minimum size=0.35cm, thick, draw, fill=white}, thick, x=0.65cm, y=0.65cm]

\draw[thin, gray!50] (-5.5,-4) to (-9.5,0) to (-5.5,4) to (-1.5,0) to (-5.5,-4) to (-5.5,4);

\node (1) at (-8.5,1) {\tiny{1}};
\node (2) at (-8.5,-1) {\tiny{2}};
\node (3) at (-7.5,2) {\tiny{3}};
\node (4) at (-7.5,0) [circle] {\tiny{4}};
\node (5) at (-7.5,-2) {\tiny{5}};
\node (6) at (-6.5,3) {\tiny{6}};
\node (7) at (-6.5,1) [circle] {\tiny{7}};
\node (8) at (-6.5,-1) [circle] {\tiny{8}};
\node (9) at (-6.5,-3) {\tiny{9}};
\node (10) at (-5.5,2) [circle] {\tiny{10}};
\node (11) at (-5.5,0) [circle] {\tiny{11}};
\node (12) at (-5.5,-2) [circle] {\tiny{12}};
\node (13) at (-4.5,3) {\tiny{13}};
\node (14) at (-4.5,1) [circle] {\tiny{14}};
\node (15) at (-4.5,-1) [circle] {\tiny{15}};
\node (16) at (-4.5,-3) {\tiny{16}};
\node (17) at (-3.5,2) {\tiny{17}};
\node (18) at (-3.5,0) [circle] {\tiny{18}};
\node (19) at (-3.5,-2) {\tiny{19}};
\node (20) at (-2.5,1) {\tiny{20}};
\node (21) at (-2.5,-1) {\tiny{21}};

\draw[->] (1) to (2);
\draw[->] (3) to (4);
\draw[->] (4) to (5);
\draw[->] (6) to (7);
\draw[->] (7) to (8);
\draw[->] (8) to (9);
\draw[->] (16) to (15);
\draw[->] (15) to (14);
\draw[->] (14) to (13);
\draw[->] (19) to (18);
\draw[->] (18) to (17);
\draw[->] (21) to (20);

\draw[->, dashed] (9) to (5);
\draw[->, dashed] (5) to (2);
\draw[->] (12) to (8);
\draw[->] (8) to (4);
\draw[->] (4) to (1);
\draw[->] (11) to (7);
\draw[->] (7) to (3);
\draw[->] (10) to (6);

\draw[->, dashed] (1) to (3);
\draw[->, dashed] (3) to (6);
\draw[->] (2) to (4);
\draw[->] (4) to (7);
\draw[->] (7) to (10);
\draw[->] (5) to (8);
\draw[->] (8) to (11);
\draw[->] (9) to (12);

\draw[->, dashed] (13) to (17);
\draw[->, dashed] (17) to (20);
\draw[->] (10) to (14);
\draw[->] (14) to (18);
\draw[->] (18) to (21);
\draw[->] (11) to (15);
\draw[->] (15) to (19);
\draw[->] (12) to (16);

\draw[->, dashed] (21) to (19);
\draw[->, dashed] (19) to (16);
\draw[->] (20) to (18);
\draw[->] (18) to (15);
\draw[->] (15) to (12);
\draw[->] (17) to (14);
\draw[->] (14) to (11);
\draw[->] (13) to (10);

\draw [shade, top color = gray!35, shading angle=-90] (-1,-0.3) to (-1,0.3) to (0.7,0.3) to (0.7,0.5) to (1.2,0) to (0.7,-0.5) to (0.7,-0.3) to (-1,-0.3);
\node [draw=none, fill=none] at (0,0) {flip};

\draw[thin, gray!50] (1.5,0) to (5.5,-4) to (9.5,0) to (5.5,4) to (1.5,0) to (9.5,0);

\node (1) at (2.5,1) {\tiny{1}};
\node (2) at (2.5,-1) {\tiny{2}};
\node (3) at (3.5,2) {\tiny{3}};
\node (4) at (3.5,0) [circle] {\tiny{4}};
\node (5) at (3.5,-2) {\tiny{5}};
\node (6) at (4.5,3) {\tiny{6}};
\node (7) at (4.5,1) [circle] {\tiny{7}};
\node (8) at (4.5,-1) [circle] {\tiny{8}};
\node (9) at (4.5,-3) {\tiny{9}};
\node (10) at (5.5,2) [circle] {\tiny{10}};
\node (11) at (5.5,0) [circle] {\tiny{11}};
\node (12) at (5.5,-2) [circle] {\tiny{12}};
\node (13) at (6.5,3) {\tiny{13}};
\node (14) at (6.5,1) [circle] {\tiny{14}};
\node (15) at (6.5,-1) [circle] {\tiny{15}};
\node (16) at (6.5,-3) {\tiny{16}};
\node (17) at (7.5,2) {\tiny{17}};
\node (18) at (7.5,0) [circle] {\tiny{18}};
\node (19) at (7.5,-2) {\tiny{19}};
\node (20) at (8.5,1) {\tiny{20}};
\node (21) at (8.5,-1) {\tiny{21}};

\draw[->] (13) to (6);
\draw[->] (17) to (10);
\draw[->] (10) to (3);
\draw[->] (20) to (14);
\draw[->] (14) to (7);
\draw[->] (7) to (1);
\draw[->] (2) to (8);
\draw[->] (8) to (15);
\draw[->] (15) to (21);
\draw[->] (5) to (12);
\draw[->] (12) to (19);
\draw[->] (9) to (16);

\draw[->, dashed] (9) to (5);
\draw[->, dashed] (5) to (2);
\draw[->] (16) to (12);
\draw[->] (12) to (8);
\draw[->] (8) to (4);
\draw[->] (19) to (15);
\draw[->] (15) to (11);
\draw[->] (21) to (18);

\draw[->, dashed] (1) to (3);
\draw[->, dashed] (3) to (6);
\draw[->] (4) to (7);
\draw[->] (7) to (10);
\draw[->] (10) to (13);
\draw[->] (11) to (14);
\draw[->] (14) to (17);
\draw[->] (18) to (20);

\draw[->, dashed] (13) to (17);
\draw[->, dashed] (17) to (20);
\draw[->] (6) to (10);
\draw[->] (10) to (14);
\draw[->] (14) to (18);
\draw[->] (3) to (7);
\draw[->] (7) to (11);
\draw[->] (1) to (4);

\draw[->, dashed] (21) to (19);
\draw[->, dashed] (19) to (16);
\draw[->] (18) to (15);
\draw[->] (15) to (12);
\draw[->] (12) to (9);
\draw[->] (11) to (8);
\draw[->] (8) to (5);
\draw[->] (4) to (2);

\end{tikzpicture}
\caption{A pair of triangles amalgamated by 1 side.}
\label{fig-flip}
\end{figure}
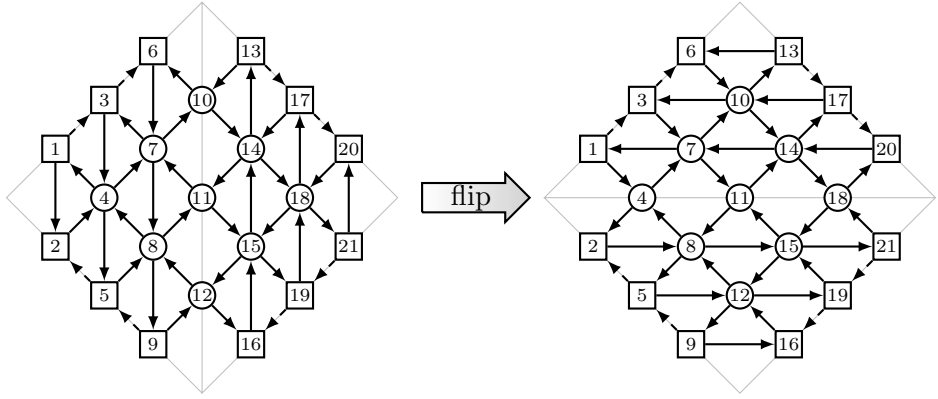

\section{Quantum groups}
\label{sect-qgroups}

In what follows, we consider the Lie algebra $\sl_{n+1}=\sl_{n+1}(\mathbb C)$ equipped with a pair of opposite Borel subalgebras $\bgt_\pm$ and a Cartan subalgebra $\hgt = \bgt_+ \cap \bgt_-$. The corresponding root system $\Delta$ is equipped with a polarization $\Delta =\Delta_+ \sqcup \Delta_-$, consistent with the choice of Borel subalgebras $\bgt_\pm$, and a set of simple roots $\hc{\alpha_1,\ldots, \alpha_n} \subset \Delta_+$. We denote by $(\cdot,\cdot)$ the unique symmetric bilinear form on $\hgt^*$ invariant under the Weyl group $W$, such that $(\alpha,\alpha)=2$ for all roots $\alpha\in \Delta$. Entries of the Cartan matrix are denoted $a_{ij} = (\alpha_i, \alpha_j)$. 

Let $q$ be a formal parameter, and consider an associative $\C(q)$-algebra $\Dgt_n$ generated by elements
$$
\hc{E_i, F_i, K_i, K'_i \,|\, i=1,\dots, n}
$$
subject to the relations
\beq
\label{K-rel}
\begin{aligned}
&K_i E_j =q^{a_{ij}} E_j K_i, &\qquad K_iK_j = K_jK_i, \\
 &K'_i E_j =q^{-a_{ij}} E_j K'_i, &\qquad K'_iK_j = K_jK'_i, \\
&K_i F_j = q^{-a_{ij}} F_j K_i, &\qquad K'_iK'_j = K'_jK'_i, \\
&K'_i F_j =q^{a_{ij}} F_j K'_i,
\end{aligned}
\eeq
the relation
\beq
\label{EF-rel}
[E_i,F_j] =\delta_{ij}\hr{q-q^{-1}}\hr{K_i-K'_i},
\eeq
and the quantum Serre relations
\beq
\label{Serre}
\begin{aligned}
&E_i^2 E_{i\pm1} - (q+q^{-1})E_i E_{i\pm1} E_i + E_{i\pm1} E_i^2 = 0, \\
&F_i^2 F_{i\pm1} - (q+q^{-1})F_i F_{i\pm1} F_i + F_{i\pm1} F_i^2 = 0, \\[2pt]
&[E_i,E_j] = [F_i, F_j] = 0 \quad\text{if}\quad |i-j|>1.
\end{aligned}
\eeq
The algebra $\Dgt_n$ is a Hopf algebra, with the comultiplication
\begin{align*}
&\Delta(E_i) = E_i\otimes 1+K_i\otimes E_i, & &\Delta(K_i) = K_i \otimes K_i, \\
&\Delta(F_i) = F_i\otimes K'_i+1\otimes F_i, & &\Delta(K'_i) = K'_i \otimes K'_i,
\end{align*}
the antipode
\begin{align*}
&S(E_i)=-K_i^{-1}E_i, & &S(K_i)=K_i^{-1}, \\
&S(F_i)=-F_iK_i, & &S(K'_i) = (K'_i)^{-1},
\end{align*}
and the counit
$$
\epsilon(K_i)=\epsilon(K'_i) = 1, \qquad \epsilon(E_i) = \epsilon(F_i)=0.
$$

The quantum group $U_q(\sl_{n+1})$ is defined as the quotient 
$$
U_q(\sl_{n+1}) = \Dgt_n / \ha{K_iK_i'=1 \,|\, i=1, \dots, n}.
$$
Note that the quantum group $U_q(\sl_{n+1})$ inherits a well-defined Hopf algebra structure from $\Dgt_n$. 
The subalgebra $U_q(\bgt) \subset \Dgt_n$ generated by all $K_i, E_i$ is a Hopf subalgebra in~$\Dgt_n$. The algebra $U_q(\bgt)$ is isomorphic to its image under the projection onto $U_q(\sl_{n+1})$ and is called the quantum Borel subalgebra of $U_q(\sl_{n+1})$. Note that $\Dgt_n$ is nothing but the Drinfeld double~\cite{Dri86} of the Hopf algebra $U_q(\bgt)$.

\begin{remark}
\label{normalization-remark}
Our definition of the quantum group generators $E_i,F_i$ differs from the standard one, in which the relation~\eqref{EF-rel} is replaced by the relation
$$
[e_i,f_j] =\delta_{ij}\frac{K_i-K'_i}{q-q^{-1}}.
$$
Our normalization
$$
E_i = (q-q^{-1}) e_i, \qquad F_i = (q-q^{-1}) f_i
$$
is chosen so that in the semiclassical $q \to 1$ limit, the quantum group delivers the Poisson algebra of functions on the Poisson-Lie dual group $SL_{n+1}^*$.	
\end{remark}

Let us fix a \emph{normal ordering} $\prec$ on $\Delta_+$, that is a total ordering such that $\alpha \prec \alpha + \beta \prec \beta$ for any $\alpha, \beta \in \Delta_+$. We set $E_{\alpha_i} = E_i$, $F_{\alpha_i} = F_i$, and define inductively
\begin{align}
\label{E-ab}
&E_{\alpha+\beta} = \frac{E_\alpha E_\beta - q^{-(\alpha,\beta)} E_\beta E_\alpha}{q-q^{-1}}, \\
\label{F-ab}
&F_{\alpha+\beta} = \frac{F_\beta F_\alpha - q^{(\alpha,\beta)} F_\alpha F_\beta}{q-q^{-1}}.
\end{align}
Then the set of all normally ordered monomials in $K_\alpha$, $K'_\alpha$, $E_\alpha$, and $F_\alpha$ for $\alpha \in \Delta_+$ forms a Poincar\'e-Birkhoff-Witt (PBW) basis for $\Dgt_n$ as a $\mathbb C(q)$-module. In what follows, we denote
$$
E_{ij} = E_{\alpha_i + \alpha_{i+1} + \ldots + \alpha_j}
\qquad\text{and}\qquad
F_{ij} = F_{\alpha_i + \alpha_{i+1} + \ldots + \alpha_j}.
$$
Finally, let us introduce for future reference the automorphism $\theta$ of the Dynkin diagram of $U_q(\sl_{n+1})$ defined by
\beq
\label{theta}
\theta(i) = n+1-i, \qquad 1\leq i\leq n.
\eeq 

\begin{figure}[h]
\centering
\begin{tikzpicture}[every node/.style={inner sep=0, minimum size=0.35cm, thick}, thick, x=1cm, y=0.55cm]

\node at (-5.6,-5) {$\mathbf V_1$};
\node at (-5.6,-3) {$\mathbf V_2$};
\node at (-5.6, 5) {$\mathbf V_n$};

\node (1) at (-5,-5) [draw] {};
\node (2) at (0,-10) [circle, draw] {};
\node (3) at (5,-5) [draw] {};
\node (4) at (-5,-3) [draw] {};
\node (5) at (-4,-4) [circle, draw] {};
\node (6) at (0,-8) [circle, draw] {};
\node (7) at (4,-4) [circle, draw] {};
\node (8) at (5,-3) [draw] {};
\node (9) at (-4,-2) [circle, draw] {};
\node (10) at (-3,-3) [circle, draw] {};
\node (11) at (0,-6) [circle, draw] {};
\node (12) at (3,-3) [circle, draw] {};
\node (13) at (4,-2) [circle, draw] {};
\node (14) at (-2,-2) [circle, draw] {};
\node (15) at (0,-4) [circle, draw] {};
\node (16) at (2,-2) [circle, draw] {};

\node (17) at (-5,-1) {};
\node (18) at (-4,-1) {};
\node (19) at (-3,-1) {};
\node (20) at (-2,-1) {};
\node (21) at (-1,-1) {};
\node (22) at (1,-1) {};
\node (23) at (2,-1) {};
\node (24) at (3,-1) {};
\node (25) at (4,-1) {};
\node (26) at (5,-1) {};

\draw [->] (1) to (2);
\draw [->] (2) to (3);

\draw [->] (4) to (5);
\draw [->] (5) to (6);
\draw [->] (6) to (7);
\draw [->] (7) to (8);

\draw [->] (17) to (9);
\draw [->] (9) to (10);
\draw [->] (10) to (11);
\draw [->] (11) to (12);
\draw [->] (12) to (13);
\draw [->] (13) to (26);

\draw [->] (19) to (14);
\draw [->] (14) to (15);
\draw [->] (15) to (16);
\draw [->] (16) to (24);

\draw[->] (3) to (7);
\draw[->] (7) to (12);
\draw[->] (12) to (16);
\draw[->] (16) to (22);

\draw[->] (8) to (13);
\draw[->] (13) to (24);

\draw [->] (21) to (14);
\draw [->] (14) to (10);
\draw [->] (10) to (5);
\draw [->] (5) to (1);

\draw[->] (19) to (9);
\draw[->] (9) to (4);

\draw[->] (22) to (15);
\draw[->] (15) to (21);

\draw[->] (23) to (16);
\draw[->] (16) to (11);
\draw[->] (11) to (14);
\draw[->] (14) to (20);

\draw[->] (24) to (12);
\draw[->] (12) to (6);
\draw[->] (6) to (10);
\draw[->] (10) to (19);

\draw[->] (25) to (13);
\draw[->] (13) to (7);
\draw[->] (7) to (2);
\draw[->] (2) to (5);
\draw[->] (5) to (9);
\draw[->] (9) to (18);

\draw [->, dashed] (1) to (4);
\draw [->, dashed] (4) to (17);

\draw [->, dashed] (26) to (8);
\draw [->, dashed] (8) to (3);

\node at (-5,0) {\vdots};
\node at (-4,0) {\vdots};
\node at (-3,0) {\vdots};
\node at (-2,0) {\vdots};
\node at (-1,0) {\vdots};
\node at (1,0) {\vdots};
\node at (2,0) {\vdots};
\node at (3,0) {\vdots};
\node at (4,0) {\vdots};
\node at (5,0) {\vdots};

\node at (5.6, 5) {$\mathbf \Lambda_1$};
\node at (5.6, 3) {$\mathbf \Lambda_2$};
\node at (5.6,-5) {$\mathbf \Lambda_n$};

\node (-1) at (5,5) [draw] {};
\node (-2) at (0,10) [circle, draw] {};
\node (-3) at (-5,5) [draw] {};
\node (-4) at (5,3) [draw] {};
\node (-5) at (4,4) [circle, draw] {};
\node (-6) at (0,8) [circle, draw] {};
\node (-7) at (-4,4) [circle, draw] {};
\node (-8) at (-5,3) [draw] {};
\node (-9) at (4,2) [circle, draw] {};
\node (-10) at (3,3) [circle, draw] {};
\node (-11) at (0,6) [circle, draw] {};
\node (-12) at (-3,3) [circle, draw] {};
\node (-13) at (-4,2) [circle, draw] {};
\node (-14) at (2,2) [circle, draw] {};
\node (-15) at (0,4) [circle, draw] {};
\node (-16) at (-2,2) [circle, draw] {};

\node (-17) at (5,1) {};
\node (-18) at (4,1) {};
\node (-19) at (3,1) {};
\node (-20) at (2,1) {};
\node (-21) at (1,1) {};
\node (-22) at (-1,1) {};
\node (-23) at (-2,1) {};
\node (-24) at (-3,1) {};
\node (-25) at (-4,1) {};
\node (-26) at (-5,1) {};

\draw [->] (-1) to (-2);
\draw [->] (-2) to (-3);

\draw [->] (-4) to (-5);
\draw [->] (-5) to (-6);
\draw [->] (-6) to (-7);
\draw [->] (-7) to (-8);

\draw [->] (-17) to (-9);
\draw [->] (-9) to (-10);
\draw [->] (-10) to (-11);
\draw [->] (-11) to (-12);
\draw [->] (-12) to (-13);
\draw [->] (-13) to (-26);

\draw [->] (-19) to (-14);
\draw [->] (-14) to (-15);
\draw [->] (-15) to (-16);
\draw [->] (-16) to (-24);

\draw[->] (-3) to (-7);
\draw[->] (-7) to (-12);
\draw[->] (-12) to (-16);
\draw[->] (-16) to (-22);

\draw[->] (-8) to (-13);
\draw[->] (-13) to (-24);

\draw [->] (-21) to (-14);
\draw [->] (-14) to (-10);
\draw [->] (-10) to (-5);
\draw [->] (-5) to (-1);

\draw[->] (-19) to (-9);
\draw[->] (-9) to (-4);

\draw[->] (-22) to (-15);
\draw[->] (-15) to (-21);

\draw[->] (-23) to (-16);
\draw[->] (-16) to (-11);
\draw[->] (-11) to (-14);
\draw[->] (-14) to (-20);

\draw[->] (-24) to (-12);
\draw[->] (-12) to (-6);
\draw[->] (-6) to (-10);
\draw[->] (-10) to (-19);

\draw[->] (-25) to (-13);
\draw[->] (-13) to (-7);
\draw[->] (-7) to (-2);
\draw[->] (-2) to (-5);
\draw[->] (-5) to (-9);
\draw[->] (-9) to (-18);

\draw [->, dashed] (-1) to (-4);
\draw [->, dashed] (-4) to (-17);

\draw [->, dashed] (-26) to (-8);
\draw [->, dashed] (-8) to (-3);

\end{tikzpicture}
\caption{$\Dc_n$-quiver.}
\label{fig-An}
\end{figure}

\section{An embedding of $U_q(\sl_{n+1})$}
\label{sect-embed}

Let us now explain how to embed $U_q(\sl_{n+1})$ into a quantum cluster $\Xc$-chart on the quantum character variety of decorated $PGL_{n+1}$-local systems on a disk $\wh S$ with a single puncture $p$, and with two marked points $x_1,x_2$ on its boundary.

We consider the ideal triangulation of $\wh S$ in which we take the pair of triangles from Figure~\ref{fig-triang} and amalgamate them by two sides as in Figure~\ref{fig-amalg}. The resulting quiver is shown in Figure~\ref{fig-An}. Note that the vertices in the central column used to be frozen before amalgamation. We shall refer to this quiver as the \emph{$\Dc_n$-quiver} and denote the corresponding quantum torus algebra by $\Dc_n$. The $\Dc_n$-quivers for $n=1$, 2, and 3 are shown in Figures~\ref{fig-A1},~\ref{fig-A2}, and~\ref{fig-A3} respectively.

Let us explain our convention for labelling the vertices of the $\Dc_n$-quiver. We denote frozen vertices in the left column by $\V_{i,-i}$ with $i=1, \dots, n$ counting South to North. Now, choose a frozen vertex $\V_{i,-i}$ and follow the arrows in the South-East direction until you hit one of the vertices in the central column. Each vertex along the way is labelled by $\V_{i,r}$, $r=-i, \dots, 0$. Then, start from the central vertex $\V_{i,0}$ and follow arrows in the North-East direction labelling vertices $\V_{i,r}$, $r=0, \dots, i$, on your way until you hit a frozen vertex in the right column, which receives the label $\V_{i,i}$. This way we label all the vertices except for the upper half of those in the central column. Now, let us rotate the $\Dc_n$-quiver by $180^\circ$, and label the image of the vertex $\V_{i,r}$ by $\La_{i,r}$. Now, we have labelled every vertex twice by some $\V$ and some $\La$ except for those in the central column. This way to label vertices, although redundant, will prove very convenient in the sequel. The following relation is easy to verify:
\beq
\label{vla}
\V_{i,\pm r} = \La_{\theta(r), \mp\theta(i)}, \qquad 1 \le r \le i \le n.
\eeq
In the above formula, $\theta$ denotes the diagram automorphism defined in~\eqref{theta}. Finally, we refer to the subset of vertices $\hc{\V_{i,r} \,|\, -i \le r < i}$ as the \emph{$\V_i$-path}. Similarly, the \emph{$\La_i$-path} is $\hc{\La_{i,r} \,|\, -i \le r < i}$.

\begin{example}
Let us refer to the $i$-th vertex in Figure~\ref{fig-A1} by $X_i$. Then the labelling suggested above is as follows:
\begin{align*}
&\V_{1,-1} = X_1, & &\V_{1,0} = X_2, & &\V_{1,1} = X_3, \\
&\La_{1,-1} = X_3, & &\La_{1,0} = X_4, & &\La_{1,1} = X_1.
\end{align*}
\end{example}

\begin{example}
Similarly, we refer to the $i$-th vertex in Figure~\ref{fig-A2} by $X_i$. Then, one has
\begin{align*}
&\V_{1,-1} = X_1, & &\V_{1,0} = X_2, & &\V_{1,1} = X_3, & &\V_{2,-2} = X_4, \\
&\V_{2,-1} = X_5, & &\V_{2,0} = X_6, & &\V_{2,1} = X_7, & &\V_{2,2} = X_8, \\
&\La_{1,-1} = X_8, & &\La_{1,0} = X_9, & &\La_{1,1} = X_4, & &\La_{2,-2} = X_3, \\
&\La_{2,-1} = X_7, & &\La_{2,0} = X_{10}, & &\La_{2,1} = X_5, & &\La_{2,2} = X_1.
\end{align*}
\end{example}

\begin{remark} As shown in~\cite{SS17a}, for any semisimple Lie algebra $\g$ the algebra $U_q(\g)$ can be embedded into the quantized algebra of global functions on the Grothendieck-Springer resolution $G\times_B B$, where $B\subset G$ is a fixed Borel subgroup in $G$. On the other hand, the variety $G\times_B B$ is isomorphic to the moduli space of $G$-local systems on the punctured disc, equipped with reduction to a Borel subgroup at the puncture, as well as a trivialization at one marked point on the boundary. Classically, this moduli space is birational to $\Xc_{\wh S,G}$, and it would be interesting to understand the precise relation between the corresponding quantizations. 
\end{remark}

We now come to the first main result of the paper.

\begin{theorem}
\label{thm-embed}
There is an embedding of algebras $\iota \colon \Dgt_n \to \Lb_n$ defined by the following assignment for $i = 1, \dots, n$:
\begin{align}
\label{E}
&E_i \longmapsto \i \sum_{r=-i}^{i-1} q^{i+r} \V_{i,-i} \V_{i,1-i} \dots \V_{i,r}, \\
\label{K}
&K_i \longmapsto q^{2i} \V_{i,-i} \V_{i,1-i} \dots \V_{i,i}, \\
\label{F}
&F_{\theta(i)} \longmapsto \i \sum_{r=-i}^{i-1} q^{i+r} \La_{i,-i} \La_{i,1-i} \dots \La_{i,r}, \\
\label{K'}
&K'_{\theta(i)} \longmapsto q^{2i} \La_{i,-i} \La_{i,1-i} \dots \La_{i,i}.
\end{align}
\end{theorem}

\begin{remark}
\label{cluster-var-rmk}
Formulas~\eqref{E} and~\eqref{F} can be rewritten as follows:
\begin{align*}
&E_i \longmapsto \i ~\mu^q_{\V_{i,i-1}} \dots \mu^q_{\V_{i,1-i}}\hr{\V_{i,-i}}, \\
&F_{\theta(i)} \longmapsto \i~ \mu^q_{\La_{i,i-1}} \dots \mu^q_{\La_{i,1-i}}\hr{\La_{i,-i}}.
\end{align*}
Thus the right hand side of the formula~\eqref{E} coincides with the cluster $\Xc$-variable corresponding to the vertex $\V_{i,-i}$, in the cluster $\Sigma_{E_i}$ obtained from the initial cluster $\Sigma$ by consecutive application of mutations at variables $\V_{i,r}$, where $r$ runs from $i-1$ to $1-i$. Similarly, the right hand side of the formula~\eqref{F} coincides with the cluster $\Xc$-variable for vertex $\La_{i,-i}$ in the cluster $\Sigma_{F_{\theta(i)}}$ obtained from $\Sigma$ by consecutive application of mutations at variables $\La_{i,r}$, where $r$ runs from $i-1$ to $1-i$. Let us also record the observation that in the cluster $\Sigma_{E_i}$, the only cluster variables adjacent to $\V_{i,-i}$ are the variable $V_{i,1-i}$ which has a single arrow pointing to $\V_{i,-i}$, the frozen variable $\Lambda_{\theta(i),-\theta(i)}$ which receives a single arrow from $V_{i,-i}$, and the frozen variables $V_{i-1,1-i}$ and $V_{i+1,-1-i}$ which receive an arrow of weight $\frac{1}{2}$ from $\V_{i,-i}$ whenever they exist. The analagous property with $\V$'s replaced by $\La$'s holds in the cluster $\Sigma_{F_\theta(i)}$. 
\end{remark}

\begin{remark}
As explained in \cite{FG09}, the quantized algebra of functions on a cluster Poisson variety comes equipped with a natural family of Hilbert space representations on which the groupoid of cluster transformations acts by unitary operators. In particular, one can restrict the Hilbert space representations of the quantum cluster algebra $\Lb_n$ to its subalgebra $U_q(\mathfrak{sl}_{n+1})$. The quantum group representations obtained in this fashion turn out to be unitary equivalent to the \emph{positive representations} defined in~\cite{FI14}, in which the quantum group acts by certain explicitly defined $q$-difference operators. We thank Ivan Ip for pointing this out to us.  
\end{remark}

\begin{example}
For $n=1$, in the notations of Figure~\ref{fig-A1}, the embedding $\iota$ reads
\begin{align*}
&E \mapsto \i X_1(1 + qX_2), & &K \mapsto q^2X_1X_2X_3, \\
& F \mapsto \i X_3(1 + qX_4), & &K' \mapsto q^2X_3X_4X_1.
\end{align*}
\end{example}

\begin{figure}[h]
\centering
\begin{tikzpicture}[every node/.style={inner sep=0, minimum size=0.4cm, draw, thick}]

\node (1) at (0,1.5) [circle] {\tiny{4}};
\node (2) at (-1.5,0) {\tiny{1}};
\node (3) at (0,-1.5) [circle] {\tiny{2}};
\node (4) at (1.5,0) {\tiny{3}};

\draw [->, thick] (1) to (2);
\draw [->, thick] (2) to (3);
\draw [->, thick] (3) to (4);
\draw [->, thick] (4) to (1);

\end{tikzpicture}
\caption{$\Dc_1$-quiver.}
\label{fig-A1}
\end{figure}

\begin{example}
For $n=2$, in the notations of Figure~\ref{fig-A2}, the embedding $\iota$ reads
\begin{align*}
&E_1 \mapsto \i X_1(1+qX_2), & K_2 \mapsto q^4X_4X_5X_6X_7X_8,& \\
& E_2 \mapsto \i X_4(1+qX_5(1+qX_6(1+qX_7))), & K_1 \mapsto q^2X_1X_2X_3,& \\
& F_1 \mapsto \i X_3(1+qX_7(1+qX_{10}(1+qX_5))), & K'_2 \mapsto q^2X_8X_9X_4,& \\
& F_2 \mapsto \i X_8(1+qX_9), & K'_1 \mapsto q^4X_3X_7X_{10}X_5X_1.&
\end{align*}
\end{example}

\begin{figure}[h]
\centering
\begin{tikzpicture}[every node/.style={inner sep=0, minimum size=0.4cm, draw, thick}, thick, y=0.8cm]

\node (1) at (-2,-1) {\tiny{1}};
\node (2) at (0,-3) [circle] {\tiny{2}};
\node (3) at (2,-1) {\tiny{3}};
\node (4) at (-2,1) {\tiny{4}};
\node (5) at (-1,0) [circle] {\tiny{5}};
\node (6) at (0,-1) [circle] {\tiny{6}};
\node (7) at (1,0) [circle] {\tiny{7}};
\node (8) at (2,1) {\tiny{8}};
\node (9) at (0,3) [circle] {\tiny{9}};
\node (10) at (0,1) [circle] {\tiny{10}};

\draw [->] (5) to (6);
\draw [->] (6) to (7);
\draw [->] (7) to (10);
\draw [->] (10) to (5);

\draw [->] (1) to (2);
\draw [->] (2) to (3);
\draw [->] (3) to (7);
\draw [->] (7) to (8);
\draw [->] (8) to (9);
\draw [->] (9) to (4);
\draw [->] (4) to (5);
\draw [->] (5) to (1);

\draw [->] (5) to (9);
\draw [->] (9) to (7);
\draw [->] (7) to (2);
\draw [->] (2) to (5);

\draw [->, dashed] (1) to (4);
\draw [->, dashed] (8) to (3);

\end{tikzpicture}
\caption{$\Dc_2$-quiver.}
\label{fig-A2}
\end{figure}

The proof of Theorem~\ref{thm-embed} will follow from Propositions~\ref{prop-hom},~\ref{prop-upper}, and~\ref{prop-inj} stated below. The first of these propositions asserts that the quantum torus algebra elements~\eqref{E} --~\eqref{K'} indeed satisfy the defining relations of the quantum group. 

\begin{prop}
\label{prop-hom}
The formulas~\eqref{E} --~\eqref{K'} define a homomorphism of algebras $\iota \colon \Dgt_n \to \Dc_n$.
\end{prop}
\begin{proof}
In what follows we abuse notations and denote an element of the algebra $\Dgt_n$ and its image under $\iota$ the same. For any $1 \le i \le n$ and $-i \le r < i$, let us define
\beq
\label{wm}
\begin{aligned}
&w_i^r=\i q^{i+r} \V_{i,-i} \dots \V_{i,r}, \\
&m_i^r=\i q^{i+r}\La_{i,-i} \dots \La_{i,r}.
\end{aligned}
\eeq
Then, the formulas~\eqref{E} --~\eqref{K'} can be rewritten as follows:
\beq
\label{gen-wm}
\begin{aligned}
&E_i = w_i^{-i} + \dots + w_i^{i-1}, & &K_i = -q w_i^{i-1} m_{\theta(i)}^{-\theta(i)}, \\
&F_{\theta(i)} = m_i^{-i} + \dots + m_i^{i-1}, & &K'_{\theta(i)} = -q m_i^{i-1} w_{\theta(i)}^{-\theta(i)}.
\end{aligned}
\eeq

It is immediate from inspecting the quiver that the relations~\eqref{K-rel} hold, as well as $[E_i, E_j] = [F_i, F_j] = 0$ for $\hm{i-j}>1$. To verify~\eqref{EF-rel} it suffices to notice that
\begin{align}
\label{wm-1}
i < \theta(j) &\implies w_i^r m_j^s = m_j^s w_i^r, \\
\label{wm-2}
i = \theta(j) &\implies w_i^r m_j^s =
\begin{cases}
q^{2} m_j^s w_i^r &\text{if} \;\; r=-i, s=j-1, \\
q^{-2} m_j^s w_i^r &\text{if} \;\; r=i-1, s=-j, \\
m_j^s w_i^r &\text{otherwise},
\end{cases} \\
\label{wm-3}
i > \theta(j) &\implies w_i^r m_j^s =
\begin{cases}
q^{2} m_j^s w_i^r &\text{if} \;\; r=\pm\theta(j), s=\mp\theta(i)-1, \\
q^{-2} m_j^s w_i^r &\text{if} \;\; s=\pm\theta(i), r=\mp\theta(j)-1, \\
m_j^s w_i^r &\text{otherwise}.
\end{cases}
\end{align}
Indeed, using formulas~\eqref{gen-wm} we can write
\beq
\label{aux-1}
\hs{E_i,F_{\theta(j)}} = \sum_{r=-i}^{i-1} \sum_{s=-j}^{j-1} \hr{w_i^r m_j^s - m_j^s w_i^r}.
\eeq
If $i<\theta(j)$ the right hand side of the equation~\eqref{aux-1} is 0 by relation~\eqref{wm-1}. Similarly, if $i > \theta(j)$, we use relation~\eqref{wm-3} and see that the right hand side of equation~\eqref{aux-1} equals
\begin{multline*}
\hr{q^2-1} \hr{m_j^{-\theta(i)-1} w_i^{\theta(j)} + m_j^{\theta(i)-1} w_i^{-\theta(j)}} \\
+ \hr{q^{-2}-1} \hr{m_j^{\theta(i)} w_i^{-\theta(j)-1} + m_j^{-\theta(i)} w_i^{\theta(j)-1}}.
\end{multline*}
On the other hand, equations~\eqref{wm} and~\eqref{vla} yield
\begin{align*}
m_j^{\theta(i)} w_i^{-\theta(j)-1}
&= -q^{2j-1} \La_{j,-j} \dots \La_{j,\theta(i)-1} \La_{j,\theta(i)} \V_{i,-i} \dots \La_{i,-\theta(j)-1} \\
&= -q^{2j-1} \La_{j,-j} \dots \La_{j,\theta(i)-1} \V_{i,-\theta(j)} \V_{i,-i} \dots \La_{i,-\theta(j)-1} \\
&= -q^{2j+1} \La_{j,-j} \dots \La_{j,\theta(i)-1} \V_{i,-i} \dots \La_{i,-\theta(j)-1} \La_{i,-\theta(j)} \\
&= q^2 m_j^{\theta(i)-1} w_i^{-\theta(j)}.
\end{align*}
Similarly,
$$
m_j^{-\theta(i)} w_i^{\theta(j)-1} = q^2 m_j^{-\theta(i)-1} w_i^{\theta(j)}
$$
and thus we conclude that
$$
\hs{E_i, F_{\theta(j)}} = 0 \qquad\text{unless}\qquad i = \theta(j).
$$
Finally, if $i = \theta(j)$ one obtains
\begin{align*}
\hs{E_i, F_{\theta(j)}}
&= \hr{q^2-1}\hr{m_{\theta(i)}^{\theta(i)-1} w_i^{-i} - m_i^{i-1} w_{\theta(i)}^{-\theta(i)}} \\
&= \hr{q-q^{-1}}\hr{K_i-K'_i}
\end{align*}
which proves formula~\eqref{EF-rel}.

Let us now check the Serre relations
$$
E_{i+1}^2E_i+E_iE_{i+1}^2=(q+q^{-1})E_{i+1}E_iE_{i+1}.
$$
Suppose $-i \le t \le i-1$ and $-i-1 \le r \le i$. We write
\beq
\label{triang}
\begin{aligned}
&t \lhd r \qquad\text{if}\qquad w_{i+1}^r w_i^t = q^{-1} w_i^t w_{i+1}^r, \\
&t \rhd r \qquad\text{if}\qquad w_{i+1}^r w_i^t = q w_i^t w_{i+1}^r.
\end{aligned}
\eeq
It is easy to verify that
$$
t \lhd r \iff
\begin{cases} 
t \le r &\text{if} \quad r < 0, \\ 
t < r &\text{if} \quad r \ge 0
\end{cases}
\qquad\text{and}\qquad
t \rhd r \iff
\begin{cases} 
t > r &\text{if} \quad r < 0, \\ 
t \ge r &\text{if} \quad r \ge 0.
\end{cases}
$$

We can now express
\begin{align*}
E_{i+1}^2 E_i
&= \sum_{r,s,t} w_{i+1}^r w_{i+1}^s w_i^t \\
&= \sum_{t \rhd r, t \rhd s} w_{i+1}^r w_{i+1}^s w_i^t + \sum_{t \rhd r, t \lhd s} w_{i+1}^r w_{i+1}^s w_i^t \\
&+ \sum_{t \lhd r, t \rhd s} w_{i+1}^r w_{i+1}^s w_i^t + \sum_{t \lhd r, t \lhd s} w_{i+1}^r w_{i+1}^s w_i^t \\
&= q \sum_{t \rhd r, t \rhd s} w_{i+1}^r w_i^t w_{i+1}^s + q^{-1} \sum_{t \rhd r, t \lhd s} w_{i+1}^r w_i^t w_{i+1}^s \\
&+ q \sum_{t \lhd r, t \rhd s} w_{i+1}^r w_i^t w_{i+1}^s + q^{-1} \sum_{t \lhd r, t \lhd s} w_{i+1}^r w_i^t w_{i+1}^s.
\end{align*}
Analogously, we have
\begin{align*}
E_i E_{i+1}^2
&= q \sum_{t \lhd r, t \rhd s} w_{i+1}^r w_i^t w_{i+1}^s + q^{-1} \sum_{t \rhd r, t \rhd s} w_{i+1}^r w_i^t w_{i+1}^s \\
&+ q \sum_{t \lhd r, t \lhd s} w_{i+1}^r w_i^t w_{i+1}^s + q^{-1} \sum_{t \rhd r, t \lhd s} w_{i+1}^r w_i^t w_{i+1}^s.
\end{align*}
Observe that if $t \rhd r$ and $t \lhd s$, then one necessarily has $r<s$. Similarly, if $t \lhd r$ and $t \rhd s$, it follows that $r>s$. We now use the fact that
\beq
\label{ww}
w_{i+1}^r w_{i+1}^s = q^2 w_{i+1}^s w_{i+1}^r \qquad\text{for}\qquad r>s
\eeq
to derive
\begin{align*}
\sum_{t \rhd r, t \lhd s} w_{i+1}^r w_i^t w_{i+1}^s &= \sum_{t \rhd r, t \lhd s} w_{i+1}^s w_i^t w_{i+1}^r, \\
\sum_{t \lhd r, t \rhd s} w_{i+1}^r w_i^t w_{i+1}^s &= \sum_{t \lhd r, t \rhd s} w_{i+1}^s w_i^t w_{i+1}^r.
\end{align*}
It therefore follows that
\begin{multline*}
E_{i+1}^2E_i+E_iE_{i+1}^2 \\
=\hr{q+q^{-1}} \hr{\sum_{t \rhd r, t \rhd s}+\sum_{t \rhd r, t \lhd s}+\sum_{t \lhd r, t \rhd s}+\sum_{t \lhd r, t \lhd s}} w_{i+1}^r w_i^t w_{i+1}^s \\
=\hr{q+q^{-1}} E_{i+1}E_iE_{i+1}.
\end{multline*}
The other nontrivial Serre relations are proved in an identical fashion.
\end{proof}

The next step in the proof of Theorem~\ref{thm-embed} is to Proposition~\ref{prop-upper} which guarantees that the image of the quantum group consists of universally Laurent elements in the quantum torus algebra $\Dc_n$. 

\begin{prop}
\label{prop-upper}
The image of $\Dgt_n$ under the homomorphism $\iota$ is contained in the algebra $\Lb_n$ of universally Laurent elements.
\end{prop}

\begin{proof}
The Proposition follows from the observations made in Remark~\ref{cluster-var-rmk} along with the following trick, which was taught to us by Linhui Shen. Let $\Sigma = (I, I_0, \eps)$ be the seed encoded by the $\Dc_n$-quiver. Consider a larger quantum cluster algebra associated to the seed $\Sigma' = (I', I'_0, \eps')$ obtained by ``framing'' the original seed $\Sigma$ as follows. First, consider a set $J_0$ together with a bijection $j: I\rightarrow J_0$, and define
$$
I'_0 = I_0 \sqcup J_0, \qquad I' = I \sqcup J_0.
$$
Then, for each vertex $i \in I$, adjoin a new frozen vertex $j(i) \in J_0$, along with a single arrow $i \to j(i)$. Thus, the exchange matrix $\eps'$ of the new seed $\Sigma'$ takes the block form
$$
\eps' =
\left[
\begin{array}{c|c}
\eps & \;\;\mathrm{Id}\;\, \\
\hline
-\mathrm{Id}\, & 0
\end{array}
\right].
$$
In particular, for any $|I'| \times |I'|$ matrix $M$, with $M_{ij}=0$ unless both $i,j$ are frozen and such that $\widetilde\eps = \eps'+M$ is an integer matrix, the corresponding cluster ensemble map $p$ induced by $p_{\Sigma'}^M$ will be an isomorphism, since the matrix $\widetilde\eps$ is invertible over~$\mathbb{Z}$.

Let us denote the quantum cluster $\Xc$-variables of the torus $\Xc^q_{\Sigma'}$ by $Y_i$ for $i \in I$ and $Y_{j(i)}$ for $j(i) \in J_0$. It is then evident from the commutation relations in the cluster tori $\Xc^q_\Sigma$ and $\Xc^q_{\Sigma'}$, that the map $\phi \colon \Xc^q_\Sigma \to \Xc^q_{\Sigma'}$ defined by $\phi(X_i) = Y_i$ is an embedding of algebras. Moreover, it is evident from the quantum $\Xc$-mutation rules~\eqref{X-mut}, that  for any sequence $\mu^q = \mu^q_{i_1} \dots \mu^q_{i_k}$ of quantum cluster mutations, one has $\mu^q\circ\phi(Y_i) = \phi\circ\mu^q(X_i)$. Thus, $\phi$ gives rise to an embedding of the corresponding quantum cluster algerbas:
$$
\phi \colon \Lb_n \hookrightarrow \Lb'_n.
$$
	
It now follows from Remark~\ref{cluster-var-rmk} that under the ensemble map $p^*$ in the cluster $\Sigma'_{E_i}$ we have
\beq
\label{Laurent}
(p^*\circ\iota)(E_i) = p^*(X_{V_{i,-i}}) = A_{V_{i,1-i}}\times m_{\mathrm{frozen}},
\eeq
where $m_{\mathrm{frozen}}$ is a monomial in the frozen $\Ac$-variables. By the quantum Laurent phenomenon~\cite{BZ05}, we know that under any sequence of quantum $\Ac$-mutations the right-hand-side of the formula~\eqref{Laurent} remains a Laurent polynomial in the quantum cluster $\Ac$-variables. Since the ensemble map $p^*$ is invertible and intertwines the quantum $\Xc$- and $\Ac$-mutations, we conclude that $\iota(E_i)\in \Lb_n$. An analogous argument applies to the image of the Chevalley generators $F_i$. Finally, one readily observes that in the initial cluster $\Sigma$, for each Cartan generator $K_i,K'_i$ the monomials $(p^*\circ\iota)(K_i)$ and $(p^*\circ\iota)(K'_i)$ only involve frozen $\Ac$-variables, and are thus manifestly universally Laurent. The Proposition is proved. 
\end{proof}

The final step of the proof of Theorem~\ref{thm-embed} is to establish that $\iota$ is indeed a faithful embedding. 

\begin{prop}
\label{prop-inj}
The homomorphism $\iota$ is injective.
\end{prop}

\begin{proof}
It will be convenient to choose a different PBW basis of $\Dgt_n$ from the one we considered in Section~\ref{sect-qgroups}. Namely, for any simple root $\alpha$ we set $F'_\alpha = F_\alpha$, then define inductively
$$
F'_{\alpha+\beta} = \frac{F'_\alpha F'_\beta - q^{-(\alpha,\beta)} F'_\beta F'_\alpha}{q-q^{-1}}
$$
for any $\alpha, \beta \in \Delta_+$ such that $\alpha \prec \beta$.To illustrate the difference between $F_\alpha$ and $F'_\alpha$, observe that one has
$$
F_{\alpha_1+\alpha_2} = \frac{F_2F_1-q^{-1}F_1F_2}{q-q^{-1}}
\qquad\text{and}\qquad
F'_{\alpha_1+\alpha_2} = \frac{F_1F_2-qF_2F_1}{q-q^{-1}}.
$$

By the PBW theorem, the set $\mathrm{Mon}_{PBW}$ of all normally ordered monomials in $K_\alpha$, $K'_\alpha$, $E_\alpha$, and $F'_\alpha$, $\alpha \in \Delta_+$, forms a basis for $\Dgt_n$ over $\mathbb C(q)$. Let us now fix a degree-lexigocraphic order on the set of all monomials in the quantum torus $\mathcal{D}_n$, taken with respect to any total order on the generators $\{X_i\}$. To establish injectivity of $\iota$, it will suffice to show that there are no two PBW monomials $m_1 ,m_2 \in \mathrm{Mon}_{PBW}$, such that $\iota(m_1)$ and $\iota(m_2)$ have the same leading term with respect to our chosen monomial order for $\mathcal{D}_n$. Indeed, if this is true, our monomial order induces a total order on $\mathrm{Mon}_{PBW}$ with respect to which the map $\iota$ becomes triangular. 

In fact, given a monomial $\vec{X}\in\mathcal{D}_n$ that arises as the leading term of some PBW monomial, one can reconstruct the unique PBW monomial $m_{\vec{X}}$ such that the leading term of $\iota(m_{\Vec{X}})$ is $\vec{X}$ as follows. In the monomial $\vec{X}$, let ${\rm n}_{ij}$, ${\rm s}_{ij}$, ${\rm e}_{ij}$, and ${\rm w}_{ij}$ be respectively the degrees of the cluster variables corresponding to North, South, East, and West nodes of the rhombus labelled by $ij$ in the right triangle in Figure~\ref{fig-amalg}; for example if $i=2$ and $j=3$, numbers ${\rm n}_{ij}$, ${\rm s}_{ij}$, ${\rm e}_{ij}$, and ${\rm w}_{ij}$ are the degrees of the cluster variables $X_5$, $X_{13}$, $X_{15}$, and $X_2$. Let us also declare ${\rm w}_{1n}=0$. Then the degree of $E_{ij}$ in $m_{\vec{X}}$ is equal to ${\rm n}_{ij} + {\rm s}_{ij} - {\rm e}_{ij} - {\rm w}_{ij}$ and the degree of $K_i$ is equal to ${\rm e}_{ii} - {\rm n}_{in}$.

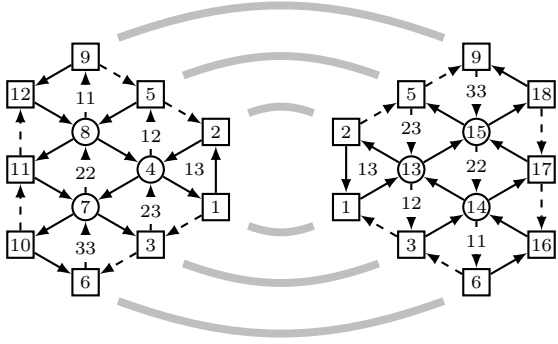
\begin{figure}[t]
\centering
\begin{tikzpicture}[every node/.style={inner sep=0, minimum size=0.35cm, thick, draw}, thick, y=0.5cm, x=0.866cm]

\node (1) at (-1,1) {\tiny{2}};
\node (2) at (-2,2) {\tiny{5}};
\node (3) at (-3,3) {\tiny{9}};
\node (4) at (-1,-1) {\tiny{1}};
\node (5) at (-2,0) [circle] {\tiny{4}};
\node (6) at (-3,1) [circle] {\tiny{8}};
\node (7) at (-4,2) {\tiny{12}};
\node (8) at (-2,-2) {\tiny{3}};
\node (9) at (-3,-1) [circle] {\tiny{7}};
\node (10) at (-4,0) {\tiny{11}};
\node (11) at (-3,-3) {\tiny{6}};
\node (12) at (-4,-2) {\tiny{10}};

\draw[->, dashed] (2) to (1);
\draw[->, dashed] (3) to (2);
\draw[->] (3) to (7);
\draw[->, dashed] (10) to (7);
\draw[->, dashed] (12) to (10);
\draw[->] (12) to (11);
\draw[->, dashed] (8) to (11);
\draw[->, dashed] (4) to (8);
\draw[->] (4) to (1);

\draw[->] (9) to (12);
\draw[->] (5) to (9);
\draw[->] (1) to (5);
\draw[->] (6) to (10);
\draw[->] (2) to (6);

\draw[->] (6) to (3);
\draw[->] (9) to (6);
\draw[->] (11) to (9);
\draw[->] (5) to (2);
\draw[->] (8) to (5);

\draw[->] (7) to (6);
\draw[->] (6) to (5);
\draw[->] (5) to (4);
\draw[->] (10) to (9);
\draw[->] (9) to (8);

\node[fill=white] at (-3,2-0.1) [draw=none] {\tiny{11}};
\node[fill=white] at (-3,0-0.1) [draw=none] {\tiny{22}};
\node[fill=white] at (-3,-2-0.1) [draw=none] {\tiny{33}};
\node[fill=white] at (-2,1-0.1) [draw=none] {\tiny{12}};
\node[fill=white] at (-2,-1-0.1) [draw=none] {\tiny{23}};
\node[fill=white] at (-1.33,0) [draw=none] {\tiny{13}};

\draw[line width = 3pt, gray!50] (-2.5,3.5) to [out=20, in=160] (2.5,3.5);
\draw[line width = 3pt, gray!50] (-1.5,2.5) to [out=20, in=160] (1.5,2.5);
\draw[line width = 3pt, gray!50] (-0.5,1.5) to [out=20, in=160] (0.5,1.5);
\draw[line width = 3pt, gray!50] (-0.5,-1.5) to [out=-20, in=-160] (0.5,-1.5);
\draw[line width = 3pt, gray!50] (-1.5,-2.5) to [out=-20, in=-160] (1.5,-2.5);
\draw[line width = 3pt, gray!50] (-2.5,-3.5) to [out=-20, in=-160] (2.5,-3.5);

\node (1) at (1,-1) {\tiny{1}};
\node (2) at (2,-2) {\tiny{3}};
\node (3) at (3,-3) {\tiny{6}};
\node (4) at (1,1) {\tiny{2}};
\node (5) at (2,0) [circle] {\tiny{13}};
\node (6) at (3,-1) [circle] {\tiny{14}};
\node (7) at (4,-2) {\tiny{16}};
\node (8) at (2,2) {\tiny{5}};
\node (9) at (3,1) [circle] {\tiny{15}};
\node (10) at (4,0) {\tiny{17}};
\node (11) at (3,3) {\tiny{9}};
\node (12) at (4,2) {\tiny{18}};

\draw[->, dashed] (2) to (1);
\draw[->, dashed] (3) to (2);
\draw[->] (3) to (7);
\draw[->, dashed] (10) to (7);
\draw[->, dashed] (12) to (10);
\draw[->] (12) to (11);
\draw[->, dashed] (8) to (11);
\draw[->, dashed] (4) to (8);
\draw[->] (4) to (1);

\draw[->] (9) to (12);
\draw[->] (5) to (9);
\draw[->] (1) to (5);
\draw[->] (6) to (10);
\draw[->] (2) to (6);

\draw[->] (6) to (3);
\draw[->] (9) to (6);
\draw[->] (11) to (9);
\draw[->] (5) to (2);
\draw[->] (8) to (5);

\draw[->] (7) to (6);
\draw[->] (6) to (5);
\draw[->] (5) to (4);
\draw[->] (10) to (9);
\draw[->] (9) to (8);

\node[fill=white] at (3,-2+0.1) [draw=none] {\tiny{11}};
\node[fill=white] at (3,0+0.1) [draw=none] {\tiny{22}};
\node[fill=white] at (3,2+0.1) [draw=none] {\tiny{33}};
\node[fill=white] at (2,-1+0.1) [draw=none] {\tiny{12}};
\node[fill=white] at (2,1+0.1) [draw=none] {\tiny{23}};
\node[fill=white] at (1.33,0) [draw=none] {\tiny{13}};

\end{tikzpicture}
\caption{A pair of triangles amalgamated by 2 sides.}
\label{fig-amalg}
\end{figure}

To see this, it suffices to analyze how much each of the generators $E_{rs}$, $F'_{rs}$, $K_r$, and $K'_r$ contribute to the expression
\beq
\label{nwse}
{\rm n}_{ij} + {\rm s}_{ij} - {\rm e}_{ij} - {\rm w}_{ij}
\eeq
for different $r$ and $s$. First, note that $K_r$ contributes 0 to all of the summands in~\eqref{nwse} if $r \ne j,j-1$. At the same time, $K_i$ contributes 1 to ${\rm s}_{ij} = {\rm w}_{i,j-1}$ and ${\rm e}_{ij} = {\rm n}_{i,j-1}$ and 0 to ${\rm n}_{ij}$, ${\rm w}_{ij}$, ${\rm s}_{i,j-1}$ and ${\rm e}_{i,j-1}$. Therefore, $K_r$ does not contribute anything to the sum~\eqref{nwse} for any $r = 1, \dots, n$. Using similar arguments one shows that the generators $K'_r$ and $F'_r$ do not contribute anything to the sum~\eqref{nwse} as well.

Now, let us consider generators $E_{rs}$. If $r <i$ or $s<j$ then none of the cluster variables corresponding to the vertices of the rhombus $ij$ appear in the expression $\iota(E_{rs})$ and therefore, $E_{rs}$ does not contribute anything to the expression~\eqref{nwse}. On the other hand, if $r > i$ then $E_{rs}$ contributes to 1 ${\rm n}_{i,j}$ if and only if it does so to ${\rm w}_{i,j}$. Similarly, it contributes simultaneously 1 or 0 to the pair ${\rm s}_{i,j}$ and ${\rm e}_{i,j}$. The same holds for $s>j$ and the pairs $({\rm n}_{i,j}, {\rm e}_{i,j})$ and $({\rm s}_{i,j}, {\rm w}_{i,j})$. Thus, it only remains to consider $r = i$ and $j = s$ in which case $E_{rs}$ contributes 1 to ${\rm s}_{i,j}$ and 0 to all other summands in~\eqref{nwse}. This finishes the proof of the formula for the exponent of $E_{ij}$ in the monomial $m_{\vec{X}}$. The formula for $K_i$ is proved by similar arguments.

%

\begin{figure}[h]
\centering
\begin{tikzpicture}[every node/.style={inner sep=0, minimum size=0.35cm, thick, draw}, thick, x=0.75cm, y=0.5cm]

\node (1) at (-3,-2) {\tiny{10}};
\node (2) at (0,-5) [circle] {\tiny{6}};
\node (3) at (3,-2) {\tiny{16}};
\node (4) at (-3,0) {\tiny{11}};
\node (5) at (-2,-1) [circle] {\tiny{7}};
\node (6) at (0,-3) [circle] {\tiny{3}};
\node (7) at (2,-1) [circle] {\tiny{14}};
\node (8) at (3,0) {\tiny{17}};
\node (9) at (-3,2) {\tiny{12}};
\node (10) at (-2,1) [circle] {\tiny{8}};
\node (11) at (-1,0) [circle] {\tiny{4}};
\node (12) at (0,-1) [circle] {\tiny{1}};
\node (13) at (1,0) [circle] {\tiny{13}};
\node (14) at (2,1) [circle] {\tiny{15}};
\node (15) at (3,2) {\tiny{18}};
\node (16) at (0,1) [circle] {\tiny{2}};
\node (17) at (0,3) [circle] {\tiny{5}};
\node (18) at (0,5) [circle] {\tiny{9}};

\draw [->] (1) to (2);
\draw [->] (2) to (3);

\draw [->] (4) to (5);
\draw [->] (5) to (6);
\draw [->] (6) to (7);
\draw [->] (7) to (8);

\draw [->] (9) to (10);
\draw [->] (10) to (11);
\draw [->] (11) to (12);
\draw [->] (12) to (13);
\draw [->] (13) to (14);
\draw [->] (14) to (15);

\draw [->] (15) to (18);
\draw [->] (18) to (9);

\draw [->] (8) to (14);
\draw [->] (14) to (17);
\draw [->] (17) to (10);
\draw [->] (10) to (4);

\draw [->] (3) to (7);
\draw [->] (7) to (13);
\draw [->] (13) to (16);
\draw [->] (16) to (11);
\draw [->] (11) to (5);
\draw [->] (5) to (1);

\draw [->] (6) to (11);
\draw [->] (11) to (17);
\draw [->] (17) to (13);
\draw [->] (13) to (6);

\draw [->] (2) to (5);
\draw [->] (5) to (10);
\draw [->] (10) to (18);
\draw [->] (18) to (14);
\draw [->] (14) to (7);
\draw [->] (7) to (2);

\draw [->, dashed] (1) to (4);
\draw [->, dashed] (4) to (9);
\draw [->, dashed] (8) to (3);
\draw [->, dashed] (15) to (8);

\end{tikzpicture}
\caption{$\Dc_3$-quiver.}
\label{fig-A3}
\end{figure}

Now, let ${\rm n}_{ij}$, ${\rm s}_{ij}$, ${\rm e}_{ij}$, and ${\rm w}_{ij}$ denote the degrees in $m_{\vec{X}}$ of the cluster variables corresponding to the North, South, East, and West nodes of the rhombus labelled $ij$ in the left triangle in Figure~\ref{fig-amalg}, where we set ${\rm e}_{1n}=0$. Then the degree of $F'_{\theta(i)\theta(j)}$ equals ${\rm n}_{ij} + {\rm s}_{ij} - {\rm e}_{ij} - {\rm w}_{ij}$ in the left triangle where we set ${\rm e}_{1n}=0$ and the degree of $K'_{\theta(i)}$ equals ${\rm w}_{i,i} - {\rm s}_{i,n}$. Again, the proofs are similar to the one for $E_{ij}$.
\end{proof}

\begin{cor}
The homomorphism $\iota$ induces an embedding of the quantum group $U_q(\sl_{n+1})$ into the quotient of the algebra $\Lb_n$ by relations
$$
q^{2n+2} \V_{i,-i} \dots \V_{i,i} \cdot \La_{\theta(i),-\theta(i)} \dots \La_{\theta(i),\theta(i)} = 1
$$
for all $1 \le i \le n$.
\end{cor}

It was shown in~\cite{GS15} that the algebra $\Lb_n$ carries a natural action of the symmetric group $S_{n+1}$ by cluster transformations. At the classical level, this action is given by permuting the $(n+1)!$ coordinate flags in $\C^{n+1}$ fixed by a generic monodromy around the puncture. We conclude this section with the following conjecture regarding the image of the homomorphism $\iota$. 

\begin{conjecture}
\label{img-conj}
The image of $\iota$ coincides with the subalgebra of elements in $\Lb_n$ invariant under the $S_{n+1}$ action, that is we have an isomorphism 
$$
\iota \colon \Dgt_n \simeq \hr{\Lb_n}^{S_{n+1}}.
$$
\end{conjecture}
We remark that for $n=1$, Conjecture~\ref{img-conj} can be proven by a direct calculation.

\section{Triangulations of a punctured disk}
\label{sect-triang}

\begin{figure}[b]
\centering
\begin{tikzpicture}[every node/.style={inner sep=0, minimum size=0.15cm, circle, draw, fill=black}, x=0.5cm, y=0.5cm]

\node [thick] (t) at (-9,2) {};
\node [thick] (c) at (-9,0) {};
\node [thick] (b) at (-9,-2) {};

\draw[thick] (-9,0) circle (2);
\draw[thick] (t) to (c);
\draw[thick] (t) to [out=-60, in=-120, min distance = 2cm] (t);

\node [thick] (t) at (0,2) {};
\node [thick] (c) at (0,0) {};
\node [thick] (b) at (0,-2) {};

\draw[thick] (0,0) circle (2);
\draw[thick] (t) to (c) to (b);

\node [thick] (t) at (9,2) {};
\node [thick] (c) at (9,0) {};
\node [thick] (b) at (9,-2) {};

\draw[thick] (9,0) circle (2);
\draw[thick] (b) to (c);
\draw[thick] (b) to [out=60, in=120, min distance = 2cm] (b);

\end{tikzpicture}
\caption{Three triangulations of the punctured disk with a pair of marked points on its boundary.}
\label{fig-triangulations}
\end{figure}
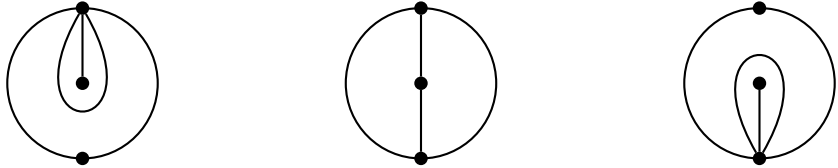

Let $\wh S$, as before, be a punctured disk with a pair of marked points on the boundary. Recall that the quiver $\Dc_n$ was constructed from a triangulation of~$\wh S$ shown in the middle of Figure~\ref{fig-triangulations}. In this section we discuss the embedding~$\iota$ in terms of clusters corresponding to the \emph{self-folded} triangulations of $\wh S$ shown in the left and right parts of Figure~\ref{fig-triangulations}. The respective quivers we denote by~$\Dc_n^{\mathrm{sf}_\pm}$, with positive sign corresponding to the triangulation in the left part of Figure~\ref{fig-triangulations}. The quiver $\Dc_3^{\mathrm{sf}_+}$ is shown in Figure~\ref{fig-sf}; it can be obtained from $\Dc_3$ by first mutating the latter at vertices 1, 3, 6, then at vertices 4, 7, 13, 14, and finally at 8, 3, 15.

\begin{figure}[h]
\centering
\begin{tikzpicture}[every node/.style={inner sep=0, minimum size=0.35cm, thick, draw, circle, fill=white}, thick, x=0.65cm, y=0.75cm]

\fill[pattern=north east lines] (-3,2) to (-2,3) to (-1,2);
\fill[pattern=north west lines] (-2,1) to (-3,2) to (-1,2) to (0,1);
\fill[pattern=north east lines] (-1,0) to (-2,1) to (0,1) to (1,0);
\fill[pattern=north west lines] (-1,2) to (-2,3) to (0,3) to (1,2);
\fill[pattern=north east lines] (0,1) to (-1,2) to (1,2) to (2,1);
\fill[pattern=north east lines] (1,2) to (0,3) to (2,3) to (3,2);

\node (10) at (-1,0) [rectangle] {\tiny{10}};
\node (16) at (1,0) [rectangle] {\tiny{16}};
\node (11) at (-2,1) [rectangle] {\tiny{11}};
\node (6) at (0,1) {\tiny{6}};
\node (17) at (2,1) [rectangle] {\tiny{17}};
\node (12) at (-3,2) [rectangle] {\tiny{12}};
\node (7) at (-1,2) {\tiny{7}};
\node (14) at (1,2) {\tiny{14}};
\node (18) at (3,2) [rectangle] {\tiny{18}};
\node (8) at (-2,3) {\tiny{8}};
\node (3) at (0,3) {\tiny{3}};
\node (15) at (2,3) {\tiny{15}};
\node (4) at (-1,4) {\tiny{4}};
\node (13) at (1,4) {\tiny{13}};
\node (1) at (0,5) {\tiny{1}};
\node (2) at (0,6) {\tiny{2}};
\node (5) at (0,7) {\tiny{5}};
\node (9) at (0,8) {\tiny{9}};

\draw [->] (10) to (16);
\draw [->] (16) to (6);
\draw [->] (6) to (7);
\draw [->] (7) to (8);
\draw [->] (8) to (12);
\draw [->] (12) to (7);
\draw [->] (7) to (14);
\draw [->] (14) to (18);
\draw [->] (18) to (15);
\draw [->] (15) to (14);
\draw [->] (14) to (6);
\draw [->] (6) to (10);

\draw [->] (11) to (6);
\draw [->] (6) to (17);
\draw [->] (17) to (14);
\draw [->] (14) to (3);
\draw [->] (3) to (7);
\draw [->] (7) to (11);
\draw [->] (11) to (6);

\draw [->] (8) to (4);
\draw [->] (4) to (3);
\draw [->] (3) to (13);
\draw [->] (13) to (15);
\draw [->] (15) to [out=75,in=-30] (9);
\draw [->] (9) to [out=-150, in=105] (8);

\draw [->] (4) to (1);
\draw [->] (1) to (13);
\draw [->] (13) to [out=90, in=-60] (5);
\draw [->] (5) to [out=-120, in=90] (4);
\draw [->] (4) to (2);
\draw [->] (2) to (13);
\draw [->] (13) to (4);
\draw [->] (4) to [out=105,in=-120] (9);
\draw [->] (9) to [out=-60, in=75] (13);

\draw [->, dashed] (10) to (11);
\draw [->, dashed] (11) to (12);
\draw [->, dashed] (18) to (17);
\draw [->, dashed] (17) to (16);

\node[fill=white] at (-.4,.6) [draw=none] {\tiny{11}};
\node[fill=white] at (.55,1.55) [draw=none] {\tiny{22}};
\node[fill=white] at (1.55,2.55) [draw=none] {\tiny{33}};
\node[fill=white] at (-1.4,1.6) [draw=none] {\tiny{12}};
\node[fill=white] at (-.45,2.55) [draw=none] {\tiny{23}};
\node[fill=white] at (-2,2.4) [draw=none] {\tiny{13}};

\end{tikzpicture}
\caption{Quiver $\Dc_3^{\mathrm{sf}_+}$.}
\label{fig-sf}
\end{figure}

We shall now express generators $E_i$ and $K_i$ of the quantum Borel subalgebra $U_q(\bgt_+) \subset U_q(\sl_{n+1})$ in terms of cluster coordinates of the quiver $\Dc_n^{\mathrm{sf}_+}$. Note that in the bottom of $\Dc_n^{\mathrm{sf}_+}$, there are $n$ horizontal rows of vertices, such that the vertices in the beginning and end of each row are frozen. Let us number the rows from bottom to top, so that there are $r+1$ vertices in the $r$-th row, and number vertices in each row from left to right, starting with 1. We denote the $j$-th vertex in the $r$-th row by $B_{r,j}$. Then it is straightforward to check that the images under the embedding $\iota$ of the generators $E_i$ and $K_i$ can be expressed as follows:
\beq
\label{EK-sf}
E_i \mapsto \i \sum_{r=1}^{i} q^{r-1} B_{i,1} B_{i,2} \dots B_{i,r},
\qquad
K_i \mapsto q^{i} B_{i,1} B_{i,2} \dots B_{i,i+1}.
\eeq
One can see that the above formulas involve fewer monomials than those for the middle triangulation in Figure~\ref{fig-triangulations}. On the other hand, formulas for the generators of the opposite Borel $U_q(\bgt_-) \subset U_q(\sl_{n+1})$ become more involved. In the cluster corresponding to the right-most triangulation in Figure~\ref{fig-triangulations} the opposite occurs: the images of $U_q(\bgt_-)$ generators under $\iota$ are simpler in terms of the quantum torus $\Dc_n^{\mathrm{sf}_-}$ than in~$\Dc_n$, while the generators of the opposite Borel $U_q(\bgt_+)$ are transformed to longer expressions. More details on the combinatorics underlying the formulas~\eqref{EK-sf}, as well as those for the generators of $U_q(\bgt_-)$, are given in~\cite[Section 4]{SS17b}.

\begin{example}
	In the notations of Figure~\ref{fig-sf}, the restriction of $\iota$ to $U_q(\bgt_+)$ reads
	\begin{align*}
	&E_1 \mapsto \i X_{10}, & K_3 \mapsto q^3X_{12}X_7X_{14}X_{18},& \\
	& E_2 \mapsto \i X_{11}(1+qX_6), & K_2 \mapsto q^2X_{11}X_6X_{17},& \\
	& E_3 \mapsto \i X_{12}(1+qX_7(1+qX_{14})), & K_1 \mapsto qX_{10}X_{16}.&
	\end{align*}
\end{example}

\begin{remark}
	Formula~\eqref{EK-sf} coincides with the ``Feigin homomorphism'', a well-known embedding of $U_q(\bgt_+)$ into a quantum torus algebra, see~\cite{Ber96,Rup15}.
\end{remark}

\begin{remark}
\label{new-hive-rmk}
	As in the proof of Proposition~\ref{prop-inj}, let $\vec{X}$ be a monomial in cluster variables of the quiver $\Dc_n^{\mathrm{sf}_+}$ that arises as the leading term of a PBW monomial in $U_q(\bgt_+)$ under the embedding $\iota$. Then one can check that the exponent of the quantum root vector $E_{ij}$ is expressed in terms of the degrees of cluster variables of $\vec{X}$ by the formula~\eqref{nwse} with respect to the shaded rhombi in Figure~\ref{fig-sf}. Define $\chi_0^t \colon \Z^N \to \Z$ with $N = {n+2\choose2}+n-1$ to be the function
\beq
\label{chi0}
\chi_0^t = \min_{1 \le i \le j \le n} \hr{ {\rm n}_{ij} + {\rm s}_{ij} - {\rm e}_{ij} - {\rm w}_{ij} },
\eeq
where the minimum is taken over all shaded rhombi in Figure~\ref{fig-sf}. The condition that a cluster monomial $\vec{X}$ appears as the leading term of a PBW basis element is thus equivalent to the condition $\chi_0^t(\vec{X})\ge0$. 
	
Let us now consider the triangle in Figure~\ref{fig-sf} whose sides contain vertices $\hc{10,11,12}$, $\hc{8,3,15}$, and $\hc{18,17,16}$ respectively. By rotating the above triangle by $\pm2\pi/3$ one obtains two more families of shaded rhombi, and we define functions $\chi_\pm$ by the same formula as~\eqref{chi0} but with respect to these new families. Finally, let us define $\Wc^t \colon \Z^N \to \Z$ by
\beq
\label{potential}
\Wc^t = \min\hc{\chi_0, \chi_+, \chi_-}.
\eeq
The function $\Wc^t$ is the tropicalized potential function on the moduli space of framed $G$-local systems on a $3$-gon, defined in~\cite{GS15} by Goncharov and Shen. As shown in~\cite{GS15}, there is a bijection between lattice points $z \in \Z^N$ such that $\Wc^t(z) \ge 0$ and the so-called ``hives'' introduced by Knutson and Tao in~\cite{KT99}.
\end{remark}

%
%
%

\section{The Dehn twist on a twice punctured disk}
\label{sect-twist}

In order to describe the coalgebra structure of $U_q(\sl_{n+1})$, we will need to consider the moduli space $\Xc_{\wh S_2,PGL_{n+1}}$ of $PGL_{n+1}$-local systems on $\wh S_2$, a disk with {\em two} punctures $p_1,p_2$, and two marked points $x_1,x_2$ on its boundary. To obtain a quantum cluster chart on this moduli space, we consider the quiver corresponding to the $(n+1)$-triangulation of the left-most disk in Figure~\ref{fig-Dehn}. Note that this quiver is formed by amalgamating two $\Dc_n$-quivers by one column of frozen variables, see Figure~\ref{fig-An}. An example of two amalgamated $\Dc_2$-quivers is shown in Figure~\ref{fig-A2o2}, where one should disregard the gray arrows. We refer to the result of this amalgamation as the $\Zc_n$-quiver and denote the corresponding quantum torus algebra by $\Zc_n$.

Figure~\ref{fig-Dehn} shows four different ideal triangulations of a twice punctured disk with two marked points on the boundary; the arrows correspond to flips of ideal triangulations. Note that the right-most disk may be obtained from the left-most one by applying the half-Dehn twist rotating the left puncture clockwise about the right one. Hence this half-Dehn twist may be decomposed into a sequence of 4 flips. Let $\Zc'_n$ be the quiver obtained from the $(n+1)$-triangulation of the right-most disk. It is evident from inspecting the corresponding $(n+1)$-triangulations that there exists an isomorphism $\sigma$ between the $\Zc_n$- and the $\Zc'_n$-quivers that preserves all frozen variables. On the other hand, since there is no nontrivial automorphism of the $\Zc_n$-quiver fixing its frozen variables, we conclude that the isomorphism $\sigma$ is unique.

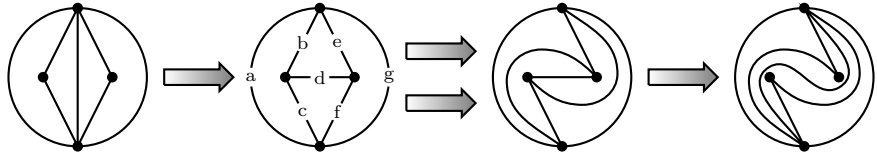
\begin{figure}[h]
\centering
\begin{tikzpicture}[every node/.style={inner sep=0, minimum size=0.25cm, thick, circle}, x=0.23cm,y=0.23cm]

\draw[thick] (-21+0,0) circle (4);
\draw[thick] (-21+0,4) to (-21-2,0) to (-21+0,-4) to (-21+2,0) to (-21+0,4) to (-21+0,-4);
\fill (-21+0,-4) circle (2pt);
\fill (-21-2,0) circle (2pt);
\fill (-21+0,4) circle (2pt);
\fill (-21+2,0) circle (2pt);

\draw[thick, shade, shading angle = -90] (-14-2,-0.45) to (-14-2,0.45) to (-14+1.2,0.45) to (-14+1.2,0.8) to (-14+2,0) to (-14+1.2,-0.8) to (-14+1.2,-0.45) to (-14-2,-0.45);

\draw[thick] (-7+0,0) circle (4);
\draw[thick] (-7+2,0) to (-7+0,4) to (-7-2,0) to (-7+0,-4) to (-7+2,0) to (-7-2,0);
\fill (-7+0,-4) circle (2pt);
\fill (-7-2,0) circle (2pt);
\fill (-7+0,4) circle (2pt);
\fill (-7+2,0) circle (2pt);

\node[fill=white] at (-11,0) {\tiny{a}};
\node[fill=white] at (-8,2) {\tiny{b}};
\node[fill=white] at (-8,-2) {\tiny{c}};
\node[fill=white] at (-7,0) {\tiny{d}};
\node[fill=white] at (-6,2) {\tiny{e}};
\node[fill=white] at (-6,-2) {\tiny{f}};
\node[fill=white] at (-3,0) {\tiny{g}};

\draw[thick, shade, shading angle = -90] (-2,1.5-0.45) to (-2,1.5+0.45) to (1.2,1.5+0.45) to (1.2,1.5+0.8) to (2,1.5+0) to (1.2,1.5-0.8) to (1.2,1.5-0.45) to (-2,1.5-0.45);

\draw[thick, shade, shading angle = -90] (-2,-1.5-0.45) to (-2,-1.5+0.45) to (1.2,-1.5+0.45) to (1.2,-1.5+0.8) to (2,-1.5+0) to (1.2,-1.5-0.8) to (1.2,-1.5-0.45) to (-2,-1.5-0.45);

\draw[thick] (7+0,0) circle (4);
\draw[thick] (7+2,0) to [out=135, in=0] (7-1.2,1.6) to [out=180, in=90] (7-3.2,-0.2) to [out=-90, in=145] (7+0,-4) to (7-2,0) to [out=-45, in=180] (7+1.2,-1.6) to [out=0, in=-90] (7+3.2,0.2) to [out=90, in=-35] (7+0,4) to (7+2,0) to (7-2,0);
\fill (7+0,-4) circle (2pt);
\fill (7-2,0) circle (2pt);
\fill (7+0,4) circle (2pt);
\fill (7+2,0) circle (2pt);

\draw[thick, shade, shading angle = -90] (14-2,-0.45) to (14-2,0.45) to (14+1.2,0.45) to (14+1.2,0.8) to (14+2,0) to (14+1.2,-0.8) to (14+1.2,-0.45) to (14-2,-0.45);

\draw[thick] (21+0,0) circle (4);
\draw[thick] (21+0,4) to (21+2,0) to [out=135, in=0] (21-1.2,1.6) to [out=180, in=90] (21-3.2,-0.2) to [out=-90, in=145] (21+0,-4) to (21-2,0) to [out=-45, in=180] (21+1.2,-1.6) to [out=0, in=-90] (21+3.2,0.2) to [out=90, in=-35] (21+0,4) to [out=-50, in=30] (21+2,-0.6) to [out=-150, in=-40] (21+0,0) [out=140, in=30] to (21-2,0.6) to [out=-150, in=130] (21+0,-4);
\fill (21+0,-4) circle (2pt);
\fill (21-2,0) circle (2pt);
\fill (21+0,4) circle (2pt);
\fill (21+2,0) circle (2pt);

\end{tikzpicture}
\caption{The half Dehn twist as a sequence of 4 flips.}
\label{fig-Dehn}
\end{figure}

Let us now describe $\sigma$ explicitly. Recall that each $(n+1)$-triangulated triangle contains exactly $n$ solid oriented paths parallel to each of its sides. We number them starting from the opposite vertex. For example, in the 4-triangulation shown in Figure~\ref{fig-triang}, the paths $1 \to 2$, $3 \to 4 \to 5$, and $6 \to 7 \to 8 \to 9$, are respectively the 1-st, the 2-nd, and the 3-rd paths parallel to the side $BC$. Now, consider the second disk in Figure~\ref{fig-Dehn}, recall that the $(n+1)$-triangulation of the pair of triangles in the middle is shown in the right part of Figure~\ref{fig-flip}. For $i = 1, \dots, n$ we define the $i$-th \emph{permutation cycle} to
\begin{itemize}
\item[$\bullet$] follow the $i$-th solid path parallel to the side $a$ in the triangle $\Delta_{abc}$ along the orientation,
\item[$\bullet$] follow the $\theta(i)$-th solid path parallel to the side $d$ in the triangle $\Delta_{bde}$ in the direction opposite to the orientation,
\item[$\bullet$] follow the $i$-th solid path parallel to the side $g$ in the triangle $\Delta_{efg}$ along the orientation,
\item[$\bullet$] follow the $\theta(i)$-th solid path parallel to the side $d$ in the triangle $\Delta_{cdf}$ in the direction opposite to the orientation.
\end{itemize}
Now, the isomorphism $\sigma$ is defined as follows: each vertex in the $i$-th permutation cycle is moved $i$ steps along the cycle, frozen variables are left intact, the rest of the vertices are rotated by $180^\circ$. In Figure~\ref{fig-A2o2}, the 2 cycles in the quiver $\Zc_2$ and the rotation of vertices 9 and 11 are shown by gray arrows; the action of $\sigma$ reads
$$
\sigma = (2 \; 7 \; 15 \; 17 \; 13 \; 4) \; (3 \; 16 \; 18) \; (8 \; 10 \; 12) \; (9 \; 11),
$$
where the 2nd permutation cycle breaks into $(3 \; 16 \; 18) \; (8 \; 10 \; 12)$.

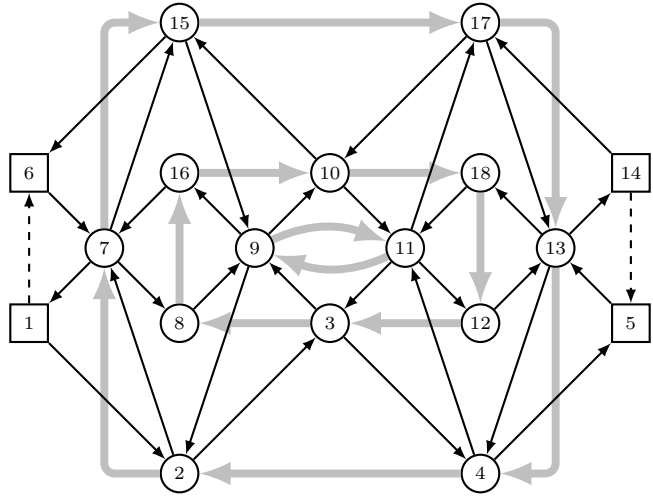
\begin{figure}[h]
\centering
\begin{tikzpicture}[every node/.style={inner sep=0, minimum size=0.5cm, thick, circle, draw}, thick]

\node (1) at (-4,-1) [rectangle] {\tiny{1}};
\node (2) at (-2,-3) {\tiny{2}};
\node (3) at (0,-1) {\tiny{3}};
\node (4) at (2,-3) {\tiny{4}};
\node (5) at (4,-1) [rectangle] {\tiny{5}};
\node (6) at (-4,1) [rectangle] {\tiny{6}};
\node (7) at (-3,0) {\tiny{7}};
\node (8) at (-2,-1) {\tiny{8}};
\node (9) at (-1,0) {\tiny{9}};
\node (10) at (0,1) {\tiny{10}};
\node (11) at (1,0) {\tiny{11}};
\node (12) at (2,-1) {\tiny{12}};
\node (13) at (3,0) {\tiny{13}};
\node (14) at (4,1) [rectangle] {\tiny{14}};
\node (15) at (-2,3) {\tiny{15}};
\node (16) at (-2,1) {\tiny{16}};
\node (17) at (2,3) {\tiny{17}};
\node (18) at (2,1) {\tiny{18}};

\draw[->, line width=3pt, gray!50, rounded corners] (2) to (-3,-3) to (7);
\draw[->, line width=3pt, gray!50, rounded corners] (7) to (-3,3) to (15);
\draw[->, line width=3pt, gray!50] (15) to (17);
\draw[->, line width=3pt, gray!50, rounded corners] (17) to (3,3) to (13);
\draw[->, line width=3pt, gray!50, rounded corners] (13) to (3,-3) to (4);
\draw[->, line width=3pt, gray!50] (4) to (2);

\draw[->, line width=3pt, gray!50] (8) to (16);
\draw[->, line width=3pt, gray!50] (16) to (10);
\draw[->, line width=3pt, gray!50] (10) to (18);
\draw[->, line width=3pt, gray!50] (18) to (12);
\draw[->, line width=3pt, gray!50] (12) to (3);
\draw[->, line width=3pt, gray!50] (3) to (8);

\draw[->, line width=3pt, gray!50] (9) to [out=25, in=160] (11);
\draw[->, line width=3pt, gray!50] (11) to [out=-155, in=-20] (9);

\draw [->] (1) to (2);
\draw [->] (2) to (3);
\draw [->] (3) to (4);
\draw [->] (4) to (5);

\draw [->] (6) to (7);
\draw [->] (7) to (8);
\draw [->] (8) to (9);
\draw [->] (9) to (10);
\draw [->] (10) to (11);
\draw [->] (11) to (12);
\draw [->] (12) to (13);
\draw [->] (13) to (14);

\draw [->] (14) to (17);
\draw [->] (17) to (10);
\draw [->] (10) to (15);
\draw [->] (15) to (6);

\draw [->] (5) to (13);
\draw [->] (13) to (18);
\draw [->] (18) to (11);
\draw [->] (11) to (3);
\draw [->] (3) to (9);
\draw [->] (9) to (16);
\draw [->] (16) to (7);
\draw [->] (7) to (1);

\draw [->] (2) to (7);
\draw [->] (7) to (15);
\draw [->] (15) to (9);
\draw [->] (9) to (2);

\draw [->] (4) to (11);
\draw [->] (11) to (17);
\draw [->] (17) to (13);
\draw [->] (13) to (4);

\draw [->, dashed] (1) to (6);
\draw [->, dashed] (14) to (5);

\end{tikzpicture}
\caption{Permutation on the $\Zc_2$-quiver.}
\label{fig-A2o2}
\end{figure}

\section{Cluster realization of the $R$-matrix}
\label{sect-R}

In this section we shall need an $\hbar$-formal version of the algebra $\Dgt_n$. Let $\overline\Dgt_n$ be an associative algebra over the ring $\C[\hbar]$ defined by generators
$$
\hc{E_i, F_i, H_i, H'_i \,|\, i=1,\dots, n}
$$
subject to the relations~\eqref{K-rel},~\eqref{EF-rel}, and~\eqref{Serre} where
$$
q = e^\hbar \qquad\text{and}\qquad K_i = q^{H_i} = e^{\hbar H_i}.
$$
Then, the $\hbar$-formal version $U_\hbar(\sl_{n+1})$ of the quantum group is defined as the quotient of the algebra $\overline\Dgt_n$ by relations
\beq
\label{H-quot}
H_i + H_i'=0
\eeq
for $i=1, \dots, n$. Note, that $\Dgt_n$ and $U_q(\sl_{n+1})$ are topological subalgebras of $\overline\Dgt_n$ and $U_\hbar(\sl_{n+1})$ respectively.

Recall that the universal $R$-matrix of $\overline\Dgt_n$ is an element
$$
\Rc \in \overline\Dgt_n \widetilde\otimes \overline\Dgt_n
$$
of a completion of its tensor square, with resect to the $\hbar$-adic topology. The image of $\Rc$ under the quotient by relations~\eqref{H-quot} gives rise to a braiding operator on the category of finite dimensional $U_\hbar(\sl_{n+1})$-modules. It was shown in~\cite{Ros89,KR90,LS91} that the universal $R$-matrix admits decomposition
$$
\Rc = \bar\Rc\Kc.
$$
where
$$
\Kc = q^{\sum_{i,j} c_{ij} H_i \otimes H'_j},
$$
and $(c_{ij})$ is the inverse of the Cartan matrix. The tensor $\bar\Rc$ is called the quasi $R$-matrix and is given by the formula
\beq
\label{quasi}
\bar\Rc = \prod_{\alpha \in \Delta_+}^{\to} \Psi^q \hr{- E_\alpha \otimes F_\alpha},
\eeq
where the product is ordered consistently with the previously chosen normal ordering $\prec$ on $\Delta_+$.

We shall also make use of the $\hbar$-formal version of the algebra $\Dc_n$. Let $\hc{X_i \,|\, i = 1, \dots, N}$ be the vertices of the $\Dc_n$-quiver, and $(\eps_{ij})$ be its incidence matrix, here $N = n(n+3)$. We define $\overline\Dc_n$ to be an algebra over $\C[\hbar]$ with generators $\hc{x_i \,|\, i = 1, \dots, N}$ subject to relations
$$
[x_i,x_j] = -2\hbar\eps_{ij}.
$$
Then, the assignment $q = e^\hbar$ and the map
$$
X_i \longmapsto e^{x_i}
$$
make $\Dc_n$ into a topological subalgebra of $\overline\Dc_n$. Combining the latter map with the embedding $\iota \colon \Dgt_n \to \Dc_n$ one gets
\begin{align}
\label{H}
&H_i \longmapsto \mathrm{v}_{i,-i} + \mathrm{v}_{i,1-i} + \dots + \mathrm{v}_{i,i}, \\
\label{H'}
&H'_{\theta(i)} \longmapsto \sla_{i,-i} + \sla_{i,1-i} + \dots + \sla_{i,i},
\end{align}
where $\V_{i,j} = \exp(\mathrm{v}_{i,j})$ and $\La_{i,j} = \exp(\sla_{ij})$.

Let $\Ad_\Kc$ and $\Ad_{\bar\Rc}$ denote the automorphisms of $\overline\Dgt_n \otimes \overline\Dgt_n$ that conjugate by $\Kc$ and $\bar\Rc$ respectively. One can check that both automorphism restrict to the topological subalgebra $\Dgt_n \otimes \Dgt_n$. It is also clear that $\Ad_{\bar\Rc}$ extends to an automorphisms of $\Dc_n \otimes \Dc_n$. Formulas~\eqref{H} and~\eqref{H'} allow one to define an action of $\Ad_\Kc$ on the algebra $\overline\Dc_n \otimes \overline\Dc_n$, which in turn restricts to an automorphisms of $\Dc_n \otimes \Dc_n$. Finally, we write $P$ for the automorphism of $\Dc_n \otimes \Dc_n$ permuting the tensor factors:
$$
P(X \otimes Y) = Y \otimes X.
$$

Recall the isomorphism of quivers described in the previous section. It defines a permutation of cluster variables $X_i \mapsto X_{\sigma(i)}$ which we also denote by $\sigma$ with a slight abuse of notation. Note that each of the 4 flips shown in Figure~\ref{fig-Dehn} corresponds to a sequence of $n+2 \choose 3$ cluster mutations, as explained at the end of Section~\ref{sect-conf}. Let
$$
N = 4 \cdot {n+2 \choose 3}
$$
and $\mu_N \dots \mu_1$ be the sequence of quantum cluster mutations constituting the half-Dehn twist. Now we are ready to formulate the next main result of the paper.

\begin{theorem}
\label{thm-main}
The composition
$$
P \circ \Ad_\Rc \colon \Dc_n \otimes \Dc_n \longra \Dc_n \otimes \Dc_n
$$
restricts to the subalgebra $\Zc_n$. Moreover, the following automorphisms of $\Zc_n$ coincide:
$$
P \circ \Ad_\Rc = \mu_N \dots \mu_1 \circ \sigma,
$$
where the sequence of quantum cluster mutations $\mu_N \dots \mu_1$ constitutes the half Dehn twist.
\end{theorem}

\begin{proof}
By Lemma~\ref{mut-decomp} we have
$$
\mu_N \dots \mu_1 = \Phi_N \circ \M_N,
$$
where $\M_N$ is a monomial transformation, and $\Phi_N$ is a conjugation by a sequence of $N$ quantum dilogarithms. The result of the theorem then follows from Proposition~\ref{lem-K} and Corollary~\ref{lem-R} below.

\end{proof}

\begin{prop}
\label{lem-K}
The following automorphisms of $\Zc_n$ coincide:
\beq
\label{Cartan-mut}
P \circ \Ad_\Kc = \M_N \circ \sigma.
\eeq
\end{prop}

\begin{proof}
We define the $\La\V_i$-path in the $\Zc_n$-quiver as the concatenation of the $\La_{\theta(i)}$-path in the left $\Dc_n$-quiver with the $\V_i$-path in the right $\Dc_n$-quiver. For example, in the notations of Figure~\ref{fig-A2o2}, the $\La\V_1$-path consists of vertices 1, 7, 16, 9, 3, 4, 5. Each mutation from the sequence $\mu_N \dots \mu_1$ happens at a vertex that belongs to a certain $\La\V_i$-path, has exactly two outgoing edges within this path, and has exactly two incoming edges from vertices that do not belong to the path. This claim can be easily verified by inspecting the $\Zc_n$-quiver and the sequence of mutations under discussion. In turn, it implies that the monomial transformation $\M_N$ restricts to each $\La\V$-path. In other words, for a cluster variable $X$ in $\La\V_i$-path the expression $\M_N(X)$ depends only on the variables from $\La\V_i$-path, and $\M_N(X) = X$ if $X$ does not belong to any $\La\V$-path. The action of $\M_N$ on the $\La\V_i$-path is shown in Figure~\ref{fig-MN}, where
\begin{align*}
&Z_- = q^{2\theta(i)} \cdot X_1X_2 \dots X_{2\theta(i)+1} \cdot Y_1, \\
&Z_0 = q^{-2n} \cdot X_{2\theta(i)}^{-1} \dots X_2^{-1}X_1^{-1} \cdot Y_{2i}^{-1} \dots Y_2^{-1}Y_1^{-1}, \\
&Z_+ = q^{2i} \cdot X_1 \cdot Y_1Y_2 \dots Y_{2i+1}.
\end{align*}

\begin{figure}[h]
\centering
\begin{tikzpicture}[every node/.style={inner sep=0, minimum size=0.35cm, thick, circle}, thick]

\node at (-5,1.5) {\tiny{$X_{2\theta(i)+1}$}};
\node at (-4,1.5) {\tiny{$X_{2\theta(i)}$}};
\node at (-3,1.5) {};
\node at (-2,1.5) {\tiny{$X_2$}};
\node at (-1,1.5) {\tiny{$X_1Y_1$}};
\node at (0,1.5) {\tiny{$Y_2$}};
\node at (1,1.5) {\tiny{$Y_3$}};
\node at (2,1.5) {};
\node at (3,1.5) {\tiny{$Y_{2i-1}$}};
\node at (4,1.5) {\tiny{$Y_{2i}$}};
\node at (5,1.5) {\tiny{$Y_{2i+1}$}};

\node[draw] (1) at (-5,1) [rectangle] {};
\node[draw] (2) at (-4,1) {};
\node (3) at (-3,1) {\tiny{\dots}};
\node[draw] (4) at (-2,1) {};
\node[draw] (5) at (-1,1) [shade, shading angle = -45] {};
\node[draw] (6) at (0,1) {};
\node[draw] (7) at (1,1) {};
\node (8) at (2,1) {\tiny{\dots}};
\node[draw] (9) at (3,1) {};
\node[draw] (10) at (4,1) {};
\node[draw] (11) at (5,1) [rectangle] {};

\draw [<-] (1) to (2);
\draw [<-] (2) to (3);
\draw [<-] (3) to (4);
\draw [<-] (4) to (5);
\draw [->] (5) to (6);
\draw [->] (6) to (7);
\draw [->] (7) to (8);
\draw [->] (8) to (9);
\draw [->] (9) to (10);
\draw [->] (10) to (11);

\draw[shade, shading angle = 180] (-0.25,0.7) to (-0.25,-0.3) to (-0.4,-0.3) to (0,-0.7) to (0.4,-0.3) to (0.25,-0.3) to (0.25,0.7) to (-0.25,0.7);

\node at (-5,-1.5) {\tiny{$Z_-$}};
\node at (-4,-1.5) {\tiny{$Y_2$}};
\node at (-3,-1.5) {\tiny{$Y_3$}};
\node at (-2,-1.5) {};
\node at (-1,-1.5) {\tiny{$Y_{2i-1}$}};
\node at (0,-1.5) {\tiny{$Y_{2i}$}};
\node at (1,-1.5) {\tiny{$Z_0$}};
\node at (2,-1.5) {\tiny{$X_{2\theta(i)}$}};
\node at (3,-1.5) {};
\node at (4,-1.5) {\tiny{$X_2$}};
\node at (5,-1.5) {\tiny{$Z_+$}};

\node[draw] (1) at (-5,-1) [rectangle] {};
\node[draw] (2) at (-4,-1) {};
\node[draw] (3) at (-3,-1) {};
\node (4) at (-2,-1) {\tiny{\dots}};
\node[draw] (5) at (-1,-1) {};
\node[draw] (6) at (0,-1) {};
\node[draw] (7) at (1,-1) [shade, shading angle = -45] {};
\node[draw] (8) at (2,-1) {};
\node (9) at (3,-1) {\tiny{\dots}};
\node[draw] (10) at (4,-1) {};
\node[draw] (11) at (5,-1) [rectangle] {};

\draw [<-] (1) to (2);
\draw [<-] (2) to (3);
\draw [<-] (3) to (4);
\draw [<-] (4) to (5);
\draw [<-] (5) to (6);
\draw [<-] (6) to (7);
\draw [->] (7) to (8);
\draw [->] (8) to (9);
\draw [->] (9) to (10);
\draw [->] (10) to (11);

\end{tikzpicture}
\caption{Action of $\M_N$ on the $\La\V_i$-path.}
\label{fig-MN}
\end{figure}
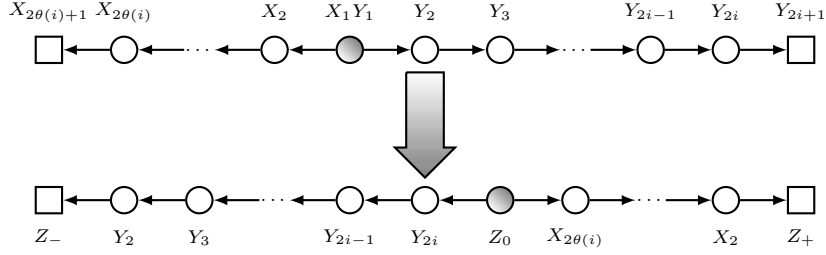

On the other hand, the automorphism $P \circ \Ad_\Kc$ acts as $P$ on all nonfrozen variables in the product $\Dgt_n \otimes \Dgt_n$. This follows from the fact that $H_i \otimes 1$ and $1 \otimes H'_j$ commute with all nonfrozen variables for all $i,j = 1, \dots, n$. It is also easy to verify that $P \circ \Ad_\Kc$ acts on the frozen variables of $\Zc_n$, and on those variables that used to be frozen before the amalgamation, as follows:
\begin{align*}
&P(\Ad_\Kc(X_{2\theta(i)+1})) = q^{2\theta(i)} \cdot X_1X_2 \dots X_{2\theta(i)+1} \cdot Y_1, \\
&P(\Ad_\Kc(X_1Y_1)) = q^{-2n} \cdot X_{2\theta(i)}^{-1} \dots X_2^{-1}X_1^{-1} \cdot Y_{2i}^{-1} \dots Y_2^{-1}Y_1^{-1}, \\
&P(\Ad_\Kc(Y_{2i+1})) = q^{2i} \cdot X_1 \cdot Y_1Y_2 \dots Y_{2i+1}.
\end{align*}
Let us prove the first equality; proofs of the other two are similar. First, note that $X_{2\theta(i)+1} = \V_{i,-i} \otimes 1$ and $Y_1 = 1 \otimes \V_{i,-i}$. Now, using formula~\eqref{H}, we get
$$
\hs{H_r \otimes H'_s, \mathrm{v}_{i,-i} \otimes 1} = a_{ir}\hbar (1 \otimes H'_s),
$$
where $(a_{ri})$ is the Cartan matrix of type $A_n$. Summing over $r$ and $s$ we obtain
$$
\hs{\sum_{r,s=1}^n c_{rs} H_r \otimes H'_s, \mathrm{v}_{i,-i} \otimes 1} = \sum\limits_{r,s=1}^n a_{ir}c_{rs} \hbar (1 \otimes H'_s) = \hbar (1 \otimes H'_i).
$$
Therefore,
$$
\Kc\hr{X_{2\theta(i)+1}}\Kc^{-1} = e^{\hbar (1 \otimes H'_i)} X_{2\theta(i)+1} = X_{2\theta(i)+1} \otimes K'_i.
$$
Finally, we note that
$$
P(X_{2\theta(i)+1}) = Y_1 \qquad\text{and}\qquad P(K'_i) = q^{2\theta(i)} X_1X_2 \dots X_{2\theta(i)+1},
$$
which yields the desired formula.
In particular, for all frozen variables $X$, we see that 
$M_N(X) = P\circ\Ad_{\mathcal{K}}(X)$.
Thus, the algebra homomorphism $M_{N}^{-1}\circ P \circ \Ad_\Kc\colon\Zc_n\rightarrow\Zc'_n$ is given by a coordinate permutation of cluster variables which fixes all frozen variables. Since there are no nontrivial automorphisms of the quivers $\Zc_n$ and $\Zc'_n$ fixing the frozen variables, we conclude that the permutation $\sigma$ must satisfy the equation~\eqref{Cartan-mut}.
\end{proof}

\begin{remark}
\label{rem-comult}
The composition of the comultiplication map $\Delta$ with the tensor square of the embedding $\iota$ factors through the subalgebra $\Zc_n$: we have
$$
(\iota \otimes \iota) \circ \Delta \colon \Dgt_n \to \Zc_n \subset \Dc_n \otimes \Dc_n.
$$
Let us refer to the concatenation of the two $\V_i$-paths in a pair of amalgamated $\Dc_n$-quivers as a $\V\V_i$-path. Then, the formula for $\Delta(E_i)$ is obtained by conjugating the first (frozen) variable in the $\V\V_i$-path by quantum dilogarithms with arguments running over consecutive vertices in the $\V\V_i$-path not including the last (frozen) vertex, and multiplying the result by $\i$. In particular, in the notations of Figure~\ref{fig-A2o2} one gets
\begin{align*}
\Delta(E_1) &= \i X_1 (1 + qX_2 (1 + qX_3(1+qX_4))), \\
\Delta(E_2) &= \i X_6 (1 + qX_7 (\dots (1+qX_{12}(1+qX_{13})) \dots )).
\end{align*}
The coproduct $\Delta(K_i)$ is equal to the product of all the variables along the $\V\V_i$-path multiplied by $q^{4i}$. Again, in the notations of Figure~\ref{fig-A2o2} one gets
\begin{align*}
\Delta(K_1) &= q^4 X_1 X_2 X_3 X_4 X_5, \\
\Delta(K_2) &= q^8 X_6 X_7 X_8 X_9 X_{10} X_{11} X_{12} X_{13} X_{14.}
\end{align*} Formulas for $\Delta(F_{\theta(i)})$ and $\Delta(K'_{\theta(i)})$ can be obtained from those for $\Delta(E_i)$ and $\Delta(K_i)$ via rotating the $\Zc_n$-quiver by $180^\circ$. Similarly, one can get formulas for iterated coproducts $\Delta^k(A)$, $A \in \Dgt_n$, by amalgamating $k+1$ copies of the $\Dc_n$-quiver.
\end{remark}

\section{Factorization of the $R$-matrix}
\label{sect-factor}

In this section, we show that the embedding~\eqref{thm-embed} gives rise to the refined factorization of the $R$-matrix of $U_q(\sl_{n+1})$ used in the proof of Theorem~\ref{thm-main}. We begin with some preparatory lemmas and remarks.

\begin{lemma}
The image of $E_{ij} \in \Dgt_n$ under the embedding $\iota$ can be written as
\beq
\label{Eij}
E_{ij} = (-q)^{j-i} \sum_{r_i \rhd \dots \rhd r_j} w_i^{r_i} w_{i+1}^{r_{i+1}} \dots w_j^{r_j}
\eeq
where the sum is taken over all tuples of integers $\hc{r_s \,|\, i \le s \le j}$ satisfying $i > r_i \rhd r_{i+1} \rhd \dots \rhd r_j \ge -j$. Similarly, the image of $F_{ij} \in \Dgt_n$ under the embedding $\iota$ can be written as
\beq
\label{Fij}
F_{ij} = \sum_{r_j \lhd \dots \lhd r_i} m_{\theta(i)}^{r_i} m_{\theta(i+1)}^{r_{i+1}} \dots m_{\theta(j)}^{r_j}
\eeq
where the sum is taken over all tuples of integers $\hc{r_s \,|\, \theta(j) \le s \le \theta(i)}$ satisfying $-\theta(j) \le r_j \lhd r_{j-1} \lhd \dots \lhd r_i < \theta(i)$.
\end{lemma}

\begin{proof}
We prove the formula~\eqref{Eij} by induction on the difference $j-i$. The base of induction follows readily from relations~\eqref{gen-wm}. Now, let us prove the step. Since
$$
E_{i,j+1} = \frac{E_{ij}E_j - qE_jE_{ij}}{q-q^{-1}}
$$
we have
\begin{multline*}
E_{i,j+1} = \frac{(-q)^{j-i}}{q-q^{-1}} \left( \sum_{r_i \rhd \dots \rhd r_j} w_i^{r_i} w_{i+1}^{r_{i+1}} \dots w_j^{r_j} \sum_{t=-j-1}^{j} w_{j+1}^t \right. \\ - \left. q \sum_{t=-j-1}^{j} w_{j+1}^t \sum_{r_i \rhd \dots \rhd r_j} w_i^{r_i} w_{i+1}^{r_{i+1}} \dots w_j^{r_j} \right).
\end{multline*}
Using notations~\eqref{triang} and the fact that $w_i^r$ and $w_j^s$ commute if $\hm{i-j}>1$, we further obtain
\begin{align*}
E_{i,j+1}
&= \frac{(-q)^{j-i}}{q-q^{-1}} \sum_{r_i \rhd \dots \rhd r_j \rhd t} \hr{1-q^2} w_i^{r_i} w_{i+1}^{r_{i+1}} \dots w_j^{r_j} w_{j+1}^t \\
&= (-q)^{j+1-i} \sum_{r_i \rhd \dots \rhd r_{j+1}} w_i^{r_i} w_{i+1}^{r_{i+1}} \dots w_{j+1}^{r_{j+1}}.
\end{align*}
This finishes the proof of~\eqref{Eij}; formula~\eqref{Fij} is proved in a similar fashion.
\end{proof}

Let us introduce a few more notations. For any fixed integer $s$ we set
\begin{align*}
&E_{i,j}^{\da s_-} = (-q)^{j-i} \sum_{\substack{s \rhd r_i \\ r_i \rhd \dots \rhd r_j}} W_{[i,j]}^{\bar r}, &
&\qquad E_{i,j}^{\da s_+} = (-q)^{j-i} \sum_{\substack{s \lhd r_i \\ r_i \rhd \dots \rhd r_j}} W_{[i,j]}^{\bar r}, \\
&E_{i,j}^{\ua s_-} = (-q)^{j-i} \sum_{\substack{r_j \lhd s \\ r_i \rhd \dots \rhd r_j}} W_{[i,j]}^{\bar r}, &
&\qquad E_{i,j}^{\ua s_+} = (-q)^{j-i} \sum_{\substack{r_j \rhd s \\ r_i \rhd \dots \rhd r_j}} W_{[i,j]}^{\bar r}.
\end{align*}
where $W_{[i,j]}^{\bar r} = w_i^{r_i} w_{i+1}^{r_{i+1}} \dots w_j^{r_j}$.
Similarly, we define
\begin{align*}
&F_{i,j}^{\da s_-} = \sum_{\substack{r_i \lhd s \\ r_j \lhd \dots \lhd r_i}} M_{[\theta(i),\theta(j)]}^{\bar r}, &
&\qquad F_{i,j}^{\da s_+} = \sum_{\substack{r_i \rhd s \\ r_j \lhd \dots \lhd r_i}} M_{[\theta(i),\theta(j)]}^{\bar r}, \\
&F_{i,j}^{\ua s_-} = \sum_{\substack{s \rhd r_j \\ r_j \lhd \dots \lhd r_i}} M_{[\theta(i),\theta(j)]}^{\bar r}, &
&\qquad F_{i,j}^{\ua s_+} = \sum_{\substack{s \lhd r_j \\ r_j \lhd \dots \lhd r_i}} M_{[\theta(i),\theta(j)]}^{\bar r}.
\end{align*}
with $M_{[\theta(i),\theta(j)]}^{\bar r} = m_{\theta(i)}^{r_i} m_{\theta(i+1)}^{r_{i+1}} \dots m_{\theta(j)}^{r_j}$.

\begin{cor}
For every $-i \le r < i$ and $-j \le s < j$ there exist decompositions
$$
E_{i+1,j} = E_{i+1,j}^{\da r_-} + E_{i+1,j}^{\da r_+},
\qquad
E_{i,j-1} = E_{i,j-1}^{\ua s_-} + E_{i,j-1}^{\ua s_+}
$$
where the summands satisfy
\beq
\label{wE}
w_i^r E_{i+1,j}^{\da r_\pm} = q^{\pm1} E_{i+1,j}^{\da r_\pm} w_i^r,
\qquad
w_j^s E_{i,j-1}^{\ua s_\pm} = q^{\pm1} E_{i,j-1}^{\ua s_\pm} w_j^s.
\eeq
Similarly, for every $-\theta(i) \le r < \theta(i)$ and $-\theta(j) \le s < \theta(j)$ there exist decompositions
$$
F_{i+1,j} = F_{i+1,j}^{\da r_-} + F_{i+1,j}^{\da r_+},
\qquad
F_{i,j-1} = F_{i,j-1}^{\ua s_-} + F_{i,j-1}^{\ua s_+}
$$
where the summands satisfy
$$
m_{\theta(i)}^r F_{i+1,j}^{\da r_\pm} = q^{\pm1} F_{i+1,j}^{\da r_\pm} m_{\theta(i)}^r,
\qquad
m_{\theta(j)}^s F_{i,j-1}^{\ua s_\pm} = q^{\pm1} F_{i,j-1}^{\ua s_\pm} m_{\theta(j)}^s.
$$
\end{cor}

\begin{cor}
For any $i<r<j$ the elements $E_{ij}$ and $F_{ij}$ can be written as follows
\begin{align}
\label{E-sum}
&E_{ij}= -q \sum_{s=-j}^{j-1} E_{i,j-1}^{\ua s_+} w_j^s
= q^2 \sum_{s=-r}^{r-1} E_{i,r-1}^{\ua s_+} w_r^s E_{r+1,j}^{\da s_-}
= -q \sum_{s=-i}^{i-1} w_i^s E_{i+1,j}^{\da s_-}, \\
\label{F-sum}
&F_{ij} = \sum_{s=-\theta(j)}^{\theta(j)-1} F_{i,j-1}^{\ua s_+} m_{\theta(j)}^s
= \sum_{s=-\theta(r)}^{\theta(r)-1} F_{i,r-1}^{\ua s_+} m_{\theta(r)}^s F_{r+1,j}^{\da s_-}
= \sum_{s=-\theta(i)}^{\theta(i)-1} m_{\theta(i)}^s F_{i+1,j}^{\da s_-}
\end{align}
\end{cor}
We say that formulas~\eqref{E-sum} show decompositions of $E_{ij}$ with respect to the $\V_i$-, $\V_r$-, and $\V_j$-paths. Similarly, formulas~\eqref{F-sum} show decompositions of $F_{ij}$ with respect to the $\La_{\theta(i)}$-, $\La_{\theta(r)}$-, and $\La_{\theta(j)}$-paths.

\begin{lemma}
\label{lem-ord}
For all $a<b$, we have
\begin{align}
\label{Ew}
\hr{E_{i,j-1}^{\ua a_+}w_j^a} \hr{E_{i,j-1}^{\ua b_+}w_j^b}
&= q^{-2} \hr{E_{i,j-1}^{\ua b_+}w_j^b} \hr{E_{i,j-1}^{\ua a_+}w_j^a}, \\
\label{Fm}
\hr{F_{i,j-1}^{\ua a_+} m_{\theta(j)}^a} \hr{F_{i,j-1}^{\ua b_+} m_{\theta(j)}^b}
&= q^{-2} \hr{F_{i,j-1}^{\ua b_+} m_{\theta(j)}^b} \hr{F_{i,j-1}^{\ua a_+} m_{\theta(j)}^a},
\end{align}
and
\begin{multline}
\label{EwE}
\hr{E_{i,j-1}^{\ua a_+}w_j^aE_{j+1,k}^{\da a_-}} \hr{E_{i,j-1}^{\ua b_+}w_j^bE_{j+1,k}^{\da b_-}} \\
=q^{-2} \hr{E_{i,j-1}^{\ua b_+}w_j^bE_{j+1,k}^{\da b_-}} \hr{E_{i,j-1}^{\ua a_+}w_j^aE_{j+1,k}^{\da a_-}}.
\end{multline}
\end{lemma}

\begin{proof}
We shall only prove the first relation as the proofs of the other two are similar. First, we note that
$$
w_j^a \hr{E_{i,j-1}^{\ua b_+} w_j^b} = q^{-1} \hr{E_{i,j-1}^{\ua b_+} w_j^b} w_j^a
$$
by equalities~\eqref{ww} and~\eqref{wE}.
Let us set
$$
E_{i,j-1}^{a \bigtriangledown b} = E_{i,j-1}^{\ua a_+} - E_{i,j-1}^{\ua b_+}.
$$
Then we have
$$
E_{i,j-1}^{\ua b_+} w_j^b = q^{-1} w_j^b E_{i,j-1}^{\ua b_+}
$$
and it only remains to commute $E_{i,j-1}^{a \bigtriangledown b}$ through $E_{i,j-1}^{\ua b_+} w_j^b$. Since
$$
E_{i,j-1}^{a \bigtriangledown b} w_j^b = q w_j^b E_{i,j-1}^{a \bigtriangledown b},
$$
it is enough to show that
$$
E_{i,j-1}^{a \bigtriangledown b} E_{i,j-1}^{\ua b_+} = q^{-2} E_{i,j-1}^{\ua b_+} E_{i,j-1}^{a \bigtriangledown b}.
$$

We finish the proof by induction on $j$. To check the base of induction we set $j=i+1$. Then
$$
E_i^{a \bigtriangledown b} = \sum_{r \rhd a, \, r \lhd b} w_i^r,
\qquad
E_i^{\ua b_+} = \sum_{s \rhd b} w_i^s
$$
and the proof follows from relation~\eqref{ww}. In order to make the step of induction we assume that~\eqref{Ew} hold for all $j<k$. Let us decompose both $E_{i,k-1}^{a \bigtriangledown b}$ and $E_{i,k-1}^{\ua b_+}$ with respect to the $\V_{k-1}$-path:
$$
E_{i,k-1}^{a \bigtriangledown b} = -q \sum_{r \rhd a, \, r \lhd b} E^{\ua r_+}_{i,k-2} w_{k-1}^r,
\qquad
E^{\ua b_+}_{i,k-1} = -q \sum_{s \rhd b} E^{\ua s_+}_{i,k-2} w_{k-1}^s.
$$
Applying~\eqref{Ew} for $j=k-1$ we conclude that it holds as well for $j=k$.
\end{proof}

\begin{lemma}
\label{lem-Serre}
For $i < j$ we have
\beq
E_{i,j}E_{k} = \begin{cases}
q^{-1}E_kE_{i,j} &\text{if} \; k=j, \\
qE_kE_{i,j} &\text{if} \; k=i,\\
E_kE_{i,j} &\text{if} \; i < k < j. \\
\end{cases}
\eeq
\end{lemma}

\begin{proof}
Using equations~\eqref{gen-wm} and~\eqref{E-sum} we can write
$$
E_{i,j}E_j = -q \sum_{r,s=-j}^{j-1} E_{i,j-1}^{\ua s_+} w_j^s w_j^r = - \sum_{r,s=-j}^{j-1} w_j^s E_{i,j-1}^{\ua s_+} w_j^r = q^{-1} E_j E_{i,j}.
$$
The other two cases are treated in a similar way.
\end{proof}

For any $i<j$ let us declare
\beq
\label{Fle}
F_{i,j}^{\ge s}= \sum_{r\ge s} F_{i,j-1}^{\ua r_+} m_{\theta(j)}^r.
\eeq
Note that
$$
F_{i,j}^{\ge s} =
\begin{cases}
F_{i,j}^{\ua s_+} &\text{if} \quad s<0, \\
F_{i,j-1}^{\ua s_+} m_{\theta(j)}^s + F_{i,j}^{\ua s_+} &\text{if} \quad s\ge0.
\end{cases}
$$
We shall also use the following shorthand:
$$
\psi(x) = \Psi^q(-x)
$$
Note that the pentagon identity~\eqref{pentagon} now reads
\beq
\label{pent-minus}
\psi(v) \psi(u) = \psi(u) \psi(-quv) \psi(v)
\eeq
for any $u$ and $v$ satisfying $vu = q^2uv$.

\begin{lemma}
\label{pent-lemma}
We have
\beq
\label{pent}
\begin{aligned}
&\psi\hr{E_{i,j} \otimes F_{i,j}^{\ua s_+}} \psi\hr{E_{i,j+1} \otimes F_{i,j+1}^{\ge s}} \psi\hr{E_{j+1} \otimes m_{\theta(j+1)}^s} \\
&\qquad=\psi\hr{E_{j+1} \otimes m_{\theta(j+1)}^s} \psi\hr{E_{i,j} \otimes F_{i,j}^{\ua s_+}} \psi\hr{E_{i,j+1} \otimes F_{i,j+1}^{\ge s+1}}.
\end{aligned}
\eeq
\end{lemma}

\begin{proof}
By definition~\eqref{Fle} and equality~\eqref{exp-prod}, there exists a factorization
$$
\psi\hr{E_{i,j+1} \otimes F_{i,j+1}^{\ge s}}
= \prod_{r\geq s} \psi\hr{E_{i,j+1} \otimes F_{i,j}^{\ua r_+} m_{\theta(j+1)}^r},
$$
where the product is taken in ascending order. Note that the dilogarithm $\psi\hr{E_{j+1} \otimes m_{\theta(j+1)}^s}$ commutes with all but the left-most factor in this product. Hence the left-hand side of~\eqref{pent} may be re-ordered so that we have a triple of adjacent factors 
$$
\psi\hr{E_{i,j} \otimes F_{i,j}^{\ua s_+}} \psi\hr{E_{i,j+1} \otimes F_{i,j}^{\ua s_+} m_{\theta(j+1)}^s} \psi\hr{E_{j+1} \otimes m_{\theta(j+1)}^s}.
$$
Let us introduce the following notations:
\begin{align*}
&A = \psi\hr{E_{i,j} \otimes F_{i,j}^{\ua s_+}}, &
&B = \psi\hr{E_{i,j+1} \otimes F_{i,j}^{\ua s_+} m_{\theta(j+1)}^s}, \\
&C = \psi\hr{E_{j+1} \otimes m_{\theta(j+1)}^s}.
\end{align*}
Now, it suffices to prove that
$$
ABC = CA.
$$
We also set
\begin{align*}
&A_r = \psi\hr{-q E_{i,j-1}^{\ua r_+} w_j^r \otimes F_{i,j}^{\ua s_+}},
&\quad
&C_r^- = \psi\hr{E_{j+1}^{\da r_-} \otimes m_{\theta(j+1)}^s}, \\
&B_r = \psi\hr{q^2 E_{i,j-1}^{\ua r_+} w_j^r E_{j+1}^{\da r_-} \otimes F_{i,j}^{\ua s_+} m_{\theta(j+1)}^s},
&\quad
&C_r^+ = \psi\hr{E_{j+1}^{\ua r_+} \otimes m_{\theta(j+1)}^s}.
\end{align*}
Formulas~\eqref{E-sum} imply
$$
A = A_{-j}A_{1-j} \dots A_{j-1},
\qquad\qquad
B = B_{-j}B_{1-j} \dots B_{j-1},
$$
and by Lemma~\ref{lem-ord}, we have
$$
A_kB_l = A_lB_k \qquad\text{for}\qquad k>l.
$$
On the other hand, for any $-j \le r \le j-1$ we can factor 
$$
C = C_r^- C_r^+.
$$
The pentagon identity~\eqref{pent-minus} yields
$$
A_rB_rC_r^- = C_r^-A_r,
$$
and therefore for any $-j \le r \le j-1$ we get
$$
A_rB_rC = A_rB_rC_r^-C_r^+ = C_r^-A_rC_r^+ = CA_r
$$
where we have used that $C_r^+$ and $A_r$ commute. Finally we obtain
\begin{align*}
ABC
&= \hr{A_{-j}A_{1-j} \dots A_{j-1}}\hr{B_{-j}B_{1-j} \dots B_{j-1}} C \\
&= \hr{A_{-j}B_{-j}} \hr{A_{1-j} B_{1-j}} \dots \hr{A_{j-1}B_{j-1}} C \\
&= C A_{-j}A_{1-j} \dots A_{j-1} \\
&= CA.
\end{align*}
\end{proof}

\begin{theorem}
\label{thm-factor}
The quasi $R$-matrix of $U_q(\sl_{n+1})$ can be factored as follows:
\beq
\label{r-factor}
\begin{aligned}
\bar\Rc_n = &\,\psi\hr{E_1 \otimes m_n^{-n}} \psi\hr{E_2 \otimes m_{n-1}^{1-n}} \cdots \psi\hr{E_{n} \otimes m_1^{-1}} \\
\cdot&\,\psi\hr{E_1 \otimes m_n^{1-n}} \psi\hr{E_2 \otimes m_{n-1}^{2-n}} \cdots \psi\hr{E_{n-1} \otimes m_2^{-1}} \\
&\;\vdots\\
\cdot&\,\psi\hr{E_1 \otimes m_n^{-2}} \psi\hr{E_2 \otimes m_{n-1}^{-1}} \\
\cdot&\,\psi\hr{E_1 \otimes m_n^{-1}} \\
\cdot&\,\psi\hr{E_1 \otimes m_n^0} \\
\cdot&\,\psi\hr{E_2 \otimes m_{n-1}^0} \psi\hr{E_1 \otimes m_n^1} \\
&\;\vdots \\
\cdot&\,\psi\hr{E_{n-1} \otimes m_2^0} \psi\hr{E_{n-2} \otimes m_3^1} \cdots \psi\hr{E_1 \otimes m_n^{n-2}} \\
\cdot&\,\psi\hr{E_n \otimes m_1^0} \psi\hr{E_{n-1} \otimes m_2^1} \cdots \psi\hr{E_1 \otimes m_n^{n-1}}.
\end{aligned}
\eeq
Equivalently, we have
\beq
\begin{aligned}
\label{R-fact1}
\bar\Rc = &\prod_{k=1}^n \prod_{j=1}^{\theta(k)} \prod_{i=-j}^{j-1} \psi\hr{ w_j^i \otimes m_{\theta(j)}^{k-\theta(j)-1}} \\
\cdot &\prod_{k=1}^n \prod_{j=\theta(k)}^{n} \prod_{i=-\theta(j)}^{\theta(j)-1} \psi\hr{ w_{\theta(j)}^i \otimes m_j^{j-\theta(k)}},
\end{aligned}
\eeq
where the products are taken in ascending order\footnote{In fact, one only needs to order the product over $k$, for the reason that all factors with a fixed $k$ commute. However, it is slightly easier to check that formulas~\eqref{R-fact1} and~\eqref{R-fact2} coincide if all three products are ordered.} and expanded from left to right, that is one should first expand the formula in $k$, then in $j$, and then in~$i$.

\end{theorem}

\begin{example}
In the case of $U_q(\sl_3)$, formula~\eqref{R-fact1} yields a factorization of the quasi $R$-matrix into the following 16 factors:
\begin{align*}
\bar\Rc = &\,\psi\hr{w_1^{-1}\otimes m_2^{-2}}\psi\hr{w_1^{0}\otimes m_2^{-2}}\psi\hr{w_2^{-2}\otimes m_1^{-1}}\psi\hr{w_2^{-1}\otimes m_1^{-1}}\\
\cdot &\,\psi\hr{w_{2}^0\otimes m_1^{-1}}\psi\hr{w_{2}^1\otimes m_1^{-1}}\psi\hr{w_{1}^{-1}\otimes m_2^{-1}}\psi\hr{w_{1}^{0}\otimes m_2^{-1}}\\
\cdot &\,\psi\hr{w_{1}^{-1}\otimes m_2^{0}}\psi\hr{w_{1}^{0}\otimes m_2^{0}}\psi\hr{w_{2}^{-2}\otimes m_1^{0}}\psi\hr{w_{2}^{-1}\otimes m_1^{0}}\\
\cdot &\,\psi\hr{w_{2}^{0}\otimes m_1^{0}}\psi\hr{w_{2}^{1}\otimes m_1^{0}}\psi\hr{w_{1}^{-1}\otimes m_2^{1}}\psi\hr{w_{1}^{0}\otimes m_2^{1}}.
\end{align*}
\end{example}

\begin{proof}
Choosing the normal ordering
$$
\alpha_1 \prec (\alpha_1 + \alpha_2) \prec (\alpha_1 + \dots + \alpha_n) \prec \alpha_2 \prec \dots \prec (\alpha_2 + \dots + \alpha_n) \prec \dots \prec \alpha_n
$$
in the formula~\eqref{quasi}, we can write
\beq
\label{ind-step}
\bar\Rc_{n+1} = \psi\hr{E_{1} \otimes F_{1}} \psi\hr{E_{1,2} \otimes F_{1,2}} \cdots \psi\hr{E_{1,n+1} \otimes F_{1,n+1}}\cdot \bar\Rc_n.
\eeq
where we may assume by induction that $\bar\Rc_n$ factors as follows:
\beq
\begin{aligned}
\label{ind-fact}
\bar\Rc_n = &\,\psi\hr{E_2 \otimes m_n^{-n}} \psi\hr{E_3 \otimes m_{n-1}^{1-n}}\cdots \psi\hr{E_{n+1} \otimes m_1^{-1}} \\
&\,\vdots \\
\cdot&\,\psi\hr{E_2 \otimes m_n^{-2}} \psi\hr{E_3 \otimes m_{n-1}^{-1}} \\
\cdot&\,\psi\hr{E_2 \otimes m_n^{-1}} \\
\cdot&\,\psi\hr{E_2 \otimes m_n^0} \\
\cdot&\,\psi\hr{E_3 \otimes m_{n-1}^0} \psi\hr{E_2 \otimes m_n^1} \\
&\,\vdots \\
\cdot&\,\psi\hr{E_n \otimes m_2^0} \psi\hr{E_{n-1} \otimes m_3^1} \cdots \psi\hr{E_2 \otimes m_n^{n-2}} \\
\cdot&\,\psi\hr{E_{n+1} \otimes m_1^0} \psi\hr{E_n \otimes m_2^1} \cdots \psi\hr{E_2 \otimes m_n^{n-1}}.
\end{aligned}
\eeq
By Lemma~\ref{lem-Serre}, we may shuffle the prefix of~\eqref{ind-step} and the first row of~\eqref{ind-fact} into the following form:
\beq
\begin{aligned}
\label{first-row}
\psi\big(E_1 \otimes F_1\big) \psi\big(E_{1,2} \,\otimes &\,F_{1,2}\big) \psi\big(E_2 \otimes m_n^{-n}\big) \dots \psi\big(E_{1,n} \otimes F_{1,n}\big) \\
\cdot\, &\psi\big(E_n \otimes m_2^{-2}\big) \psi\big(E_{1,n+1} \otimes F_{1,n+1}\big) \psi\big(E_{n+1} \otimes m_1^{-1}\big).
\end{aligned}
\eeq
We can then apply Lemma~\ref{pent-lemma} to write
\begin{align*}
&\psi\Big(E_1 \otimes F_1\Big) \psi\Big(E_{1,2} \otimes F_{1,2}\Big) \psi\Big(E_2 \otimes m_n^{-n}\Big) \\
&=\psi\Big(E_1 \otimes m_{n+1}^{-n-1}\Big) \psi\Big(E_1 \otimes F_1^{\ua (-n)_+}\Big) \psi\Big(E_{1,2} \otimes F_{1,2}\Big) \psi\Big(E_2 \otimes m_n^{-n}\Big) \\
&=\psi\Big(E_1 \otimes m_{n+1}^{-n-1}\Big) \psi\Big(E_2 \otimes m_n^{-n}\Big) \psi\Big(E_1 \otimes F_1^{\ua (-n)_+}\Big) \psi\Big(E_{1,2} \otimes F_{1,2}^{\ua (1-n)_+}\Big).
\end{align*}
After repeated applications of Lemma~\ref{pent-lemma}, the last of these being to write
\begin{multline*}
\psi\Big(E_{1,n} \otimes F_{1,n}^{\ua (-1)_+}\Big) \psi\Big(E_{1,n+1} \otimes F_{1,n+1}\Big) \psi\Big(E_{n+1} \otimes m_1^{-1}\Big) \\
=\psi\Big(E_{n+1} \otimes m_1^{-1}\Big) \psi\Big(E_{1,n} \otimes F_{1,n}^{\ua (-1)_+}\Big) \psi\Big(E_{1,n+1} \otimes F_{1,n+1}^{\ge0}\Big),
\end{multline*}
we arrive at the following form of~\eqref{first-row}:
\begin{multline*}
\psi\Big(E_1 \otimes m_{n+1}^{-n-1}\Big) \psi\Big(E_2 \otimes m_n^{-n}\Big) \cdots \psi\Big(E_{n+1} \otimes m_1^{-1}\Big) \\
\cdot\psi\Big(E_1 \otimes F_1^{\ua (-n)_+}\Big) \cdots \psi\Big(E_{1,n} \otimes F_{1,n}^{\ua (-1)_+}\Big) \psi\Big(E_{1,n+1} \otimes F_{1,n+1}^{\ge0}\Big).
\end{multline*}
We can now repeat this reasoning for each of the next $n-1$ rows in the product~\eqref{ind-fact}. This results in an expression for $\bar\Rc_{n+1}$ of the form
\begin{align*}
\bar\Rc_{n+1} = &\,\psi\Big(E_1 \otimes m_{n+1}^{-n-1}\Big) \psi\Big(E_2 \otimes m_n^{-n}\Big) \cdots \psi\Big(E_{n+1} \otimes m_1^{-1}\Big) \\
\cdot&\,\psi\Big(E_1 \otimes m_{n+1}^{-n}\Big) \psi\Big(E_2 \otimes m_n^{-n+1}\Big) \cdots \psi\Big(E_n \otimes m_2^{-1}\Big) \\
&\;\vdots \\
\cdot&\,\psi\Big(E_1 \otimes m_{n+1}^{-2}\Big) \psi\Big(E_2 \otimes m_n^{-1}\Big) \\
\cdot&\,\psi\Big(E_1 \otimes m_{n+1}^{-1}\Big) \\
\cdot&\,\psi\Big(E_1 \otimes m_{n+1}^0\Big) \\
\cdot&\,\psi\Big(E_1 \otimes F_1^{\ua 0_+}\Big) \psi\Big(E_{1,2} \otimes F_{1,2}^{\ge 0}\Big) \cdots \psi\Big(E_{1,n+1} \otimes F_{1,n+1}^{\ge 0}\Big) \\
\cdot&\,\psi\Big(E_2 \otimes m_n^0\Big) \\
\cdot&\,\psi\Big(E_3 \otimes m_{n-1}^0\Big) \psi\Big(E_2 \otimes m_n^1\Big) \\
&\;\vdots \\
\cdot&\,\psi\Big(E_{n+1} \otimes m_1^0\Big) \psi\Big(E_n \otimes m_2^1\Big) \cdots \psi\Big(E_2 \otimes m_n^0\Big).
\end{align*}

Note that the first $n+2$ rows of factors in this product are now in the desired form. 
Now we need to focus on the following factor:
\begin{align*}
&\,\psi\Big(E_1 \otimes F_1^{\ua 0_+}\Big) \psi\Big(E_{1,2} \otimes F_{1,2}^{\ge 0}\Big) \cdots \psi\Big(E_{1,n+1} \otimes F_{1,n+1}^{\ge 0}\Big) \\
\cdot&\,\psi\Big(E_2 \otimes m_n^0\Big) \\
\cdot&\,\psi\Big(E_3 \otimes m_{n-1}^0\Big) \psi\Big(E_2 \otimes m_n^1\Big) \\
&\,\vdots \\
\cdot&\,\psi\Big(E_{n+1} \otimes m_1^0\Big) \psi\Big(E_n \otimes m_2^1\Big) \cdots \psi\Big(E_2 \otimes m_n^0\Big).
\end{align*}
By Lemma~\ref{lem-Serre}, we can reshuffle this block so that it begins with an adjacent triple of terms
\begin{align*}
\psi\Big(E_1 \otimes F_1^{\ua 0_+}\Big) &\psi\Big(E_{1,2} \otimes F_{1,2}^{\ge0}\Big) \psi\Big(E_2 \otimes m_n^0\Big) \\
=\,&\psi\Big(E_2 \otimes m_n^0\Big) \psi\Big(E_1 \otimes F_1^{\ua 0_+}\Big) \psi\Big(E_{1,2} \otimes F_{1,2}^{\ua 0_+}\Big) \\
=\,&\psi\Big(E_2 \otimes m_n^0\Big) \psi\Big(E_1 \otimes m_{n+1}^1\Big) \psi\Big(E_1 \otimes F_1^{\ua 1_+}\Big) \psi\Big(E_{1,2} \otimes F_{1,2}^{\ua 0_+}\Big),
\end{align*}
where we once again used Lemma~\ref{pent-lemma}. Note that now this recovers the correct form of row $(n+3)$ in~\eqref{r-factor}, continuing in a similar fashion one arrives at the desired expression for $\bar\Rc$. 
\end{proof}

\begin{cor}
\label{lem-R}
Let $\mu_N \dots \mu_1$ be the sequence of mutations from Theorem~\eqref{thm-main}, and
$$
\mu_N \dots \mu_1 = \Phi_N \circ \M_N
$$
be the decomposition from Lemma~\ref{mut-decomp}. Then the following automorphisms of $\Zc_n$ coincide:
$$
\Ad_{P(\bar\Rc)} = \Phi_N.
$$
\end{cor}

\begin{proof}
Consider the factorization~\eqref{R-fact1} of the quasi $R$-matrix obtained in Theorem~\ref{thm-factor}. On the other hand, we have a different factorization of the $R$-matrix from inspecting the sequence of flips realizing the Dehn twist along with the corresponding sequence of mutations. The latter factorization reads
\beq
\label{R-fact2}
\begin{aligned}
P\hr{\bar\Rc} = &\prod_{k=0}^{n-1} \prod_{j=\theta(k)}^{n+1} \prod_{i=1}^{\theta(k+1)} \psi\hr{m_{j-i}^{i-\theta(k)} \otimes w_{i+\theta(j)}^{-i}} \\
\cdot &\prod_{k=0}^{n-1} \prod_{j=\theta(k)}^{n+1} \prod_{i=1}^{\theta(k+1)} \psi\hr{m_{j-i}^{i-\theta(k)} \otimes w_{i+\theta(j)}^{\theta(j)}} \\
\cdot &\prod_{k=1}^n \prod_{j=k+1}^{n+1} \prod_{i=1}^{k} \psi\hr{m_{i+\theta(j)}^{i-1} \otimes w_{j-i}^{k-j}} \\
\cdot &\prod_{k=1}^n \prod_{j=k+1}^{n+1} \prod_{i=1}^{k} \psi\hr{m_{i+\theta(j)}^{i-1} \otimes w_{j-i}^{k-i}},
\end{aligned}
\eeq
where all three products are taken in ascending order and expanded from left to right. Each row in the formula~\eqref{R-fact2} corresponds to one of the 4 flips constituting the half Dehn twist. For example, the first row can be expanded as
\begin{align*}
&\prod_{i=1}^{n} \psi\hr{m_{n+1-i}^{i-(n+1)} \otimes w_i^{-i}} \\
\cdot &\prod_{i=1}^{n-1} \bigg( \psi\hr{m_{n-i}^{i-n} \otimes w_{i+1}^{-i}} \psi\hr{m_{n+1-i}^{i-n} \otimes w_i^{-i}} \bigg) \\
&\vdots \\
\cdot &\prod_{i=1}^{1} \bigg( \psi\hr{m_1^{-1} \otimes w_n^{-1}} \psi\hr{m_2^{-1} \otimes w_{n-1}^{-1}} \dots \psi\hr{m_n^{-1} \otimes w_1^{-1}} \bigg).
\end{align*}
Note that the above formula consists of $n$ rows, while for every $i = 1, \dots, n$ the $i$-th row is a product of $i(n-i)$ quantum dilogarithms. We leave it as an exercise to an interested reader to check that the $i$-th row corresponds to the $i$-th step of the flip, as described in the last paragraph of Section~\ref{sect-conf}.

Now, it suffices to show that formulas~\eqref{R-fact1} and~\eqref{R-fact2} coincide. Let us write $\hr{a_1,\ldots,a_N}$ for the sequence of dilogarithm arguments appearing in the factorization~\eqref{R-fact1}, read from left to right. Similarly, we write $\hr{b_1,b_2,\ldots,b_N}$ for the sequence of dilogarithm arguments appearing in the factorization~\eqref{R-fact2}, again read from left to right. It is easy to see that the underlying sets $\hr{a_1,\ldots,a_N}$ and $\hr{b_1,\ldots,b_N}$ coincide. Moreover, we claim that for every pair $(b_i, b_j)$ with $i<j$ such that $(b_i, b_j)=(a_k,a_l)$ for some $k>l$, we have $[b_i,b_j]=0$. This follows from commutation relations
\begin{align*}
&w_i^r w_i^s = q^{2\sgn(r-s)} w_i^s w_i^r, \\
&w_i^r w_j^s = w_j^s w_i^r \quad\text{if}\quad \hm{i-j}>1, \\
&w_i^r w_{i+1}^s =
\begin{cases}
q w_{i+1}^s w_i^r &\text{if} \;\; r \lhd s, \\
q^{-1} w_{i+1}^s w_i^r &\text{if} \;\; r \rhd s,
\end{cases}
\end{align*}
and similar relations for variables $m_i^r$, all of which can be read from the $\Dc_n$-quiver. Hence one can freely re-order the dilogarithms $\psi(b_i)$ to match the order arising in~\eqref{R-fact1}, and the Proposition is proved.
\end{proof}

\section{Comparison with Faddeev's results}
\label{sect-example}
We conclude by comparing the rank 1 case of our results with Faddeev's embedding~\eqref{faddeev-embed} as promised in the introduction. Consider the quiver in Figure~\ref{fig-A1}. The corresponding quantum cluster $\Dc_1$ has initial variables $\ha{X_1, X_2, X_3, X_4}$ subject to the relations
$$
X_iX_{i+1} = q^{-2}X_{i+1}X_i \quad\text{and}\quad X_iX_{i+2} = X_{i+2}X_i \quad\text{where}\quad i \in \Z/4\Z.
$$
In this case, the embedding~\eqref{thm-embed} takes the form
\begin{align*}
&E \mapsto \i X_1(1+qX_2), \qquad K \mapsto q^2X_1X_2X_3, \\
&F \mapsto \i X_3(1+qX_4), \qquad K' \mapsto q^2X_3X_4X_1,
\end{align*}
while our formula~\eqref{r-factor} for the universal $R$-matrix reads
$
\Rc = \bar\Rc \Kc,
$
with
\beq
\label{R-fact}
\begin{aligned}
\bar\Rc = \Psi^q\hr{X_1 \otimes X_3} &\Psi^q\hr{qX_1 \otimes X_3X_4} \cdot \\
& \Psi^q\hr{qX_1X_2 \otimes X_3} \Psi^q\hr{q^2X_1X_2 \otimes X_3X_4}.
\end{aligned}
\eeq
Hence, Faddeev's formulas~\eqref{faddeev-embed} and~\eqref{fad-R-fact} are recovered from ours under the monomial change of variables
\beq
w_1\mapsto X_1, \quad w_2 \mapsto qX_1X_2,\quad w_3\mapsto X_3, \quad w_4\mapsto qX_3X_4.
\eeq

\end{document}